\documentclass{amsart}%
\usepackage{eurosym}
\usepackage{amsmath}
\usepackage{amssymb}
\usepackage{amsfonts}
\usepackage{graphicx}%
\newtheorem{theorem}{Theorem}[section]
\theoremstyle{plain}

\newtheorem{lemma}[theorem]{Lemma}

\newtheorem{proposition}[theorem]{Proposition}
\newtheorem{remark}[theorem]{Remark}

\numberwithin{equation}{section}
\begin{document}
\title[A static supercritical KGMP system on a closed manifold]{Semiclassical states for a static supercritical Klein-Gordon-Maxwell-Proca
system on a closed Riemannian manifold}
\author{M\'{o}nica Clapp}
\address{Instituto de Matem\'{a}ticas, Universidad Nacional Aut\'{o}noma de M\'{e}xico,
Circuito Exterior CU, 04510 M\'{e}xico DF, Mexico}
\email{monica.clapp@im.unam.mx}
\author{Marco Ghimenti}
\address{Dipartimento di Matematica Applicata, Universit\`{a} di Pisa, Via Buonarroti
1/c 56127, Pisa, Italy}
\email{marco.ghimenti@dma.unipi.it}
\author{Anna Maria Micheletti}
\address{Dipartimento di Matematica Applicata, Universit\`{a} di Pisa, Via Buonarroti
1/c 56127, Pisa, Italy}
\email{a.micheletti@dma.unipi.it}
\thanks{Research supported by CONACYT grant 129847 and UNAM-DGAPA-PAPIIT grant
IN106612 (Mexico), and MIUR projects PRIN2009: ``Variational and Topological
Methods in the Study of Nonlinear Phenomena" and ``Critical Point Theory and
Perturbative Methods for Nonlinear Differential Equations". (Italy).}
\subjclass[2010]{}
\date{\today}
\maketitle

\begin{abstract}
We establish the existence of semiclassical states for a nonlinear
Klein-Gordon-Maxwell-Proca system in static form, with Proca mass $1,$ on a
closed Riemannian manifold.

Our results include manifolds of arbitrary dimension and allow supercritical
nonlinearities. In particular, we exhibit a large class of $3$-dimensional
manifolds on which the system has semiclassical solutions for every exponent
$p\in(2,\infty).$ The solutions we obtain concentrate at closed submanifolds
of positive dimension as the singular perturbation parameter goes to cero.

\end{abstract}

\section{Introduction}

Let $(\mathfrak{M},\mathfrak{g})$ be a closed (i.e. compact and without
boundary) smooth Riemannian manifold of dimension $m\geq2$. Given real numbers
$\varepsilon>0,$ $q>0,$ $\omega\in\mathbb{R}$ and $p\in(2,\infty),$ and a
real-valued $\mathcal{C}^{1}$-function $\alpha$ such that $\alpha
(x)>\omega^{2}$ on $\mathfrak{M}$, we consider the system
\begin{equation}
\left\{
\begin{array}
[c]{cc}%
-\varepsilon^{2}\Delta_{\mathfrak{g}}\mathfrak{u}+\alpha(x)\mathfrak{u}%
=\mathfrak{u}^{p-1}+\omega^{2}(q\mathfrak{v}-1)^{2}\mathfrak{u} & \text{on
}\mathfrak{M}\text{,}\\
-\Delta_{\mathfrak{g}}\mathfrak{v}+(1+q^{2}\mathfrak{u}^{2})\mathfrak{v}%
=q\mathfrak{u}^{2} & \text{on }\mathfrak{M}\text{,}\\
\mathfrak{u},\mathfrak{v}\in H_{\mathfrak{g}}^{1}(\mathfrak{M}),\text{\quad
}\mathfrak{u},\mathfrak{v}>0. &
\end{array}
\right.  \label{eq:P-1}%
\end{equation}
The space $H_{\mathfrak{g}}^{1}(\mathfrak{M})$ is the completion of
$\mathcal{C}^{\infty}(\mathfrak{M})$ with respect to the norm defined by
$\Vert v\Vert_{\mathfrak{g}}^{2}:=\int_{\mathfrak{M}}(\left\vert
\nabla_{\mathfrak{g}}v\right\vert ^{2}+v^{2})d\mu_{\mathfrak{g}}$.

Solutions to this system correspond to standing waves of a
Klein-Gordon-Maxwell-Proca (KGMP) system in static form (i.e. one in which the
external Proca field is time-independent) with Proca mass $1$.

KGMP-systems are massive versions of the more classical electrostatic
Klein-Gordon-Maxwell (KGM)\ systems: KGM-systems are KGMP-systems with Proca
mass $0,$ i.e. the second equation in (\ref{eq:P-1})\ is replaced by
\[
-\Delta_{\mathfrak{g}}\mathfrak{v}+q^{2}\mathfrak{u}^{2}\mathfrak{v}%
=q\mathfrak{u}^{2}.
\]
Note that $\mathfrak{v}=1/q$ solves this last equation and reduces the
KGM-system to a single Schr\"{o}dinger equation in $\mathfrak{u}.$ So for the
system on a closed manifold the Proca formalism is more interesting and more
appropriate. We refer to \cite{HW}\ for a detailed discussion on KGMP-systems
and their physical meaning.

For $\varepsilon=1$ existence of solutions to system (\ref{eq:P-1}), which are
stable with respect to the phase $\omega,$ was established by Druet and Hebey
\cite{DH} and Hebey and Truong \cite{HT} for manifolds of dimension $m=3$ and
$4,$ and subcritical ($2<p<\frac{2m}{m-2}$) or critical ($p=\frac{2m}{m-2}%
$)\ nonlinearities, under certain assumptions. For critical systems in
dimension $3$\ Hebey and Wei \cite{HW} showed the existence of standing waves
with multispike amplitudes, which are unstable with respect to the phase, i.e.
they blow up with $k$ singularities as the phase $\omega$ aproaches some phase
$\omega_{0}.$

Here we are interested in semiclassical states, i.e. in solutions to system
(\ref{eq:P-1})\ for $\varepsilon$ small. The existence of semiclassical states
for similar systems in flat domains $\Omega$ in $\mathbb{R}^{m}$ has been
investigated e.g. in \cite{DW1,DW2,R}. On closed $3$-dimensional manifolds,
the existence of semiclassical states to system (\ref{eq:P-1}), which
concentrate at a single point as $\varepsilon\rightarrow0$, was established in
\cite{GMkg} and \cite{GMP} for subcritical exponents $p\in(2,6).$

The results we present in this paper apply to manifolds of arbitrary dimension
and include supercritical nonlinearities $p>2_{m}^{\ast},$ where $2_{m}^{\ast
}:=\frac{2m}{m-2}$ is the critical Sobolev exponent in dimension $m\geq3$ and
$2_{2}^{\ast}:=\infty.$ In particular, we shall exhibit a large class of
$3$-dimensional manifolds on which the system (\ref{eq:P-1}) has semiclassical
solutions for every exponent $p\in(2,\infty).$ The solutions $\mathfrak{u}$ we
obtain concentrate at closed submanifolds of $\mathfrak{M}$ of positive
dimension. Moreover, for fixed $\varepsilon$, they are stable with respect to
the phase in the sense of \cite{DH}.

Our approach consists in reducing system (\ref{eq:P-1}) to a system of a
similar type on a manifold $M$ of lower dimension but with the same exponent
$p.$ This way, if $n:=\dim M<\dim\mathfrak{M}=:m$ and $p\in\lbrack2_{m}^{\ast
},2_{n}^{\ast}),$ then $p$ is subcritical for the new system but it is
critical or supercritical for the original one. Moreover, solutions of the new
system which concentrate at a point in $M$ as $\varepsilon\rightarrow0$ will
give rise to solutions of the original system concentrating at a closed
submanifold of $\mathfrak{M}$\ of dimension $m-n$ as $\varepsilon
\rightarrow0.$

This approach was introduced by Ruf and Srikanth in \cite{RS}, where a Hopf
map is used to obtain the reduction. Reductions may also be performed by means
of other maps which preserve the Laplace-Beltrami operator, or by considering
warped products, or by a combination of both, see \cite{CGM, RS2} and the
references therein. We describe these reductions in the following two subsections.

\subsection{Warped products}

If $(M,g)$ and $(N,h)$ are closed smooth Riemannian manifolds of dimensions
$n$ and $k$ respectively, and $f:M\rightarrow(0,\infty)$ is a $\mathcal{C}%
^{1}$-map, the \emph{warped product} $M\times_{f^{2}}N$ is the cartesian
product $M\times N$ equipped with the Riemannian metric $\mathfrak{g}%
:=g+f^{2}h$.

For example, if $M$ is a closed Riemannian submanifold of $\mathbb{R}^{\ell
}\times\left(  0,\infty\right)  ,$ then%
\[
\mathfrak{M}:=\{(y,z)\in\mathbb{R}^{\ell}\times\mathbb{R}^{k+1}:(y,\left\vert
z\right\vert )\in M\},
\]
with the induced euclidian metric, is isometric to the warped product
$M\times_{f^{2}}\mathbb{S}^{k}$, where $\mathbb{S}^{k}$ is the standard
$k$-sphere and $f(x_{1},\ldots,x_{\ell+1})=x_{\ell+1}.$

Let $\pi_{M}:M\times_{f^{2}}N\rightarrow M$ be the projection. A
straightforward computation gives the following result, cf. \cite{DLD}.

\begin{proposition}
\label{prop:wp}Let $\beta:M\rightarrow\mathbb{R}$ and $\alpha=\beta\circ
\pi_{M}.$ Then $u_{\varepsilon},v_{\varepsilon}:M\rightarrow\mathbb{R}$ solve
\begin{equation}
\left\{
\begin{array}
[c]{cc}%
-\varepsilon^{2}\text{\emph{div}}_{g}\left(  f^{k}(x)\nabla_{g}u\right)
+f^{k}(x)\beta(x)u=f^{k}(x)u^{p-1}+\omega^{2}f^{k}(x)(qv-1)^{2}u & \text{on
}M,\\
-\text{\emph{div}}_{g}\left(  f^{k}(x)\nabla_{g}v\right)  +f^{k}(x)\left(
1+qu^{2}\right)  v=qf^{k}(x)u^{2} & \text{on }M,
\end{array}
\right.  \label{eq:div}%
\end{equation}
iff $\mathfrak{u}_{\varepsilon}:=u_{\varepsilon}\circ\pi_{M},$ $\mathfrak{v}%
_{\varepsilon}:=v_{\varepsilon}\circ\pi_{M}:M\times_{f^{2}}N\rightarrow
\mathbb{R}$\ solve
\begin{equation}
\left\{
\begin{array}
[c]{cc}%
-\varepsilon^{2}\Delta_{\mathfrak{g}}\mathfrak{u}+\alpha(x)\mathfrak{u}%
=\mathfrak{u}^{p-1}+\omega^{2}(q\mathfrak{v}-1)^{2}\mathfrak{u} & \text{on
}M\times_{f^{2}}N,\\
-\Delta_{\mathfrak{g}}\mathfrak{v}+\left(  1+q\mathfrak{u}^{2}\right)
\mathfrak{v}=q\mathfrak{u}^{2} & \text{on }M\times_{f^{2}}N.
\end{array}
\right.  \label{eq:wp}%
\end{equation}

\end{proposition}

Note that the exponent $p$ is the same for both systems. So if $p\in
(2_{n+k}^{\ast},2_{n}^{\ast})$ then $p$ is subcritical for (\ref{eq:div}) but
supercritical for (\ref{eq:wp}). Moreover, if the functions $u_{\varepsilon}$
concentrate at a point $\xi_{0}\in M$ as $\varepsilon\rightarrow0$, then the
functions $\mathfrak{u}_{\varepsilon}:=u_{\varepsilon}\circ\pi_{M}$
concentrate at the submanifold $\pi_{M}^{-1}(\xi_{0})\cong(N,f^{2}(\xi_{0})h)$
as $\varepsilon\rightarrow0$.

\subsection{Harmonic morphisms}

\label{subsec:hm}Let $(\mathfrak{M},\mathfrak{g})$ and $(M,g)$ be closed
Riemannian manifolds of dimensions $m$ and $n$ respectively. A \emph{harmonic
morphism} is a horizontally conformal submersion $\pi:\mathfrak{M}\rightarrow
M$ with dilation $\lambda:\mathfrak{M}\rightarrow\lbrack0,\infty)$ which
satisfies%
\begin{equation}
(n-2)\mathcal{H}(\nabla_{\mathfrak{g}}\ln\lambda)+(m-n)\kappa^{\mathcal{V}}=0,
\label{eq:hm}%
\end{equation}
where $\kappa^{\mathcal{V}}$ is the mean curvature of the fibers of $\pi$ and
$\mathcal{H}$ is the projection of the tangent space of $\mathfrak{M}$ onto
the space orthogonal to the fibers, see \cite{bw}.

So for $n=2$ a harmonic morphism is just a horizontally conformal submersion
$\pi:\mathfrak{M}\rightarrow M$ with minimal fibers. Typical examples are the
Hopf fibration $\mathbb{S}^{3}\rightarrow\mathbb{S}^{2}$ whose fiber is
$\mathbb{S}^{1}$, and the induced fibration $\mathbb{R}P^{3}\rightarrow
\mathbb{S}^{2}$ with fiber $\mathbb{R}P^{1}$, see \cite[Example 2.4.15]{bw}.
They are, in fact, Riemannian submersions (i.e. $\lambda\equiv1$).

Harmonic morphisms preserve the Laplace-Beltrami operator, i.e.%
\[
\Delta_{\mathfrak{g}}(u\circ\pi)=\lambda^{2}\left[  (\Delta_{g}u)\circ
\pi\right]
\]
for every $\mathcal{C}^{2}$-function $u:M\rightarrow\mathbb{R}.$ This fact
yields the following result.

\begin{proposition}
\label{prop:hm}Assume there exist $\beta:M\rightarrow\mathbb{R}$ and
$\mu:M\rightarrow(0,\infty)$ such that $\beta\circ\pi=\alpha$ and $\mu\circ
\pi=\lambda^{2}$. Then $u_{\varepsilon},v_{\varepsilon}:M\rightarrow
\mathbb{R}$ solve the system
\begin{equation}
\left\{
\begin{array}
[c]{cc}%
-\varepsilon^{2}\Delta_{{g}}u+\frac{\beta(x)}{\mu(x)}u=\frac{1}{\mu(x)}%
u^{p-1}+\frac{\omega^{2}}{\mu(x)}(qv-1)^{2}u & \text{on }M,\\
-\Delta_{g}v+\frac{1}{\mu(x)}\left(  1+qu^{2}\right)  v=\frac{q}{\mu(x)}u^{2}
& \text{on }M,
\end{array}
\right.  \label{eq:M}%
\end{equation}
iff $\mathfrak{u}_{\varepsilon}:=u_{\varepsilon}\circ\pi_{M},$ $\mathfrak{v}%
_{\varepsilon}:=v_{\varepsilon}\circ\pi_{M}:\mathfrak{M}\rightarrow\mathbb{R}$
solve the system
\begin{equation}
\left\{
\begin{array}
[c]{cc}%
-\varepsilon^{2}\Delta_{\mathfrak{g}}\mathfrak{u}+\alpha(x)\mathfrak{u}%
=\mathfrak{u}^{p-1}+\omega^{2}(q\mathfrak{v}-1)^{2}\mathfrak{u} & \text{on
}\mathfrak{M},\\
-\Delta_{\mathfrak{g}}\mathfrak{v}+\left(  1+q\mathfrak{u}^{2}\right)
\mathfrak{v}=q\mathfrak{u}^{2} & \text{on }\mathfrak{M}.
\end{array}
\right.  \label{eq:Mfrac}%
\end{equation}

\end{proposition}

Again, if $p\in(2_{m}^{\ast},2_{n}^{\ast}),$ the system (\ref{eq:M}) is
subcritical and the system (\ref{eq:Mfrac}) is supercritical and, if the
functions $u_{\varepsilon}$ concentrate at a point $\xi_{0}\in M$ as
$\varepsilon\rightarrow0$, the functions $\mathfrak{u}_{\varepsilon
}:=u_{\varepsilon}\circ\pi_{M}$ concentrate at the $(m-n)$-dimensional
submanifold $\pi_{M}^{-1}(\xi_{0})\ $of $\mathfrak{M}$ as $\varepsilon
\rightarrow0.$

\subsection{The main result for the general system}

Propositions \ref{prop:wp} and \ref{prop:hm} suggest studying a more general KGMP-system.

Let $(M,g)$ be a closed Riemannian manifold of dimension $n=2$ or $3$,
$\ a,b,c\in\mathcal{C}^{1}(M,\mathbb{R})$ be strictly positive functions,
$\varepsilon,q\in(0,\infty),$ $p\in(2,2_{n}^{\ast}),$ and $\omega\in
\mathbb{R}$ be such that $a(x)>\omega^{2}b(x)$ on $M$. We consider the
subcritical system
\begin{equation}
\left\{
\begin{array}
[c]{cc}%
-\varepsilon^{2}\text{div}_{g}\left(  c(x)\nabla_{g}u\right)
+a(x)u=b(x)u^{p-1}+b(x)\omega^{2}(qv-1)^{2}u & \text{in }M,\\
-\text{div}_{g}\left(  c(x)\nabla_{g}v\right)  +b(x)(1+q^{2}u^{2})v=b(x)qu^{2}
& \text{in }M,\\
u,v\in H_{g}^{1}(M),\quad u,v>0. &
\end{array}
\right.  \label{eq:P}%
\end{equation}

\begin{theorem}
\label{thm:main}Let $K$ be a $\mathcal{C}^{1}$-stable critical set of the
function $\Gamma:M\rightarrow\mathbb{R}$ given by
\[
\Gamma(x):=\frac{c(x)^{\frac{n}{2}}a(x)^{\frac{p}{p-2}-\frac{n}{2}}%
}{b(x)^{\frac{2}{p-2}}}.
\]
Then, for $\varepsilon$ small enough, the system \emph{(\ref{eq:P})} has a
solution $(u_{\varepsilon},v_{\varepsilon})$ such that $u_{\varepsilon}$
concentrates at a point $\xi_{0}\in K$ as $\varepsilon\rightarrow0.$
\end{theorem}

Recall that $K$ is a $\mathcal{C}^{1}$\emph{-stable critical set} of a
function $f\in\mathcal{C}^{1}(M,\mathbb{R})$ if $K\subset\left\{  x\in
M:\nabla_{g}f(x)=0\right\}  $ and for any $\mu>0$ there exists $\delta>0$ such
that, if $h\in\mathcal{C}^{1}(M,\mathbb{R})$ with
\[
\max_{d_{g}(x,K)\leq\mu}|f(x)-h(x)|+|\nabla_{g}f(x)-\nabla_{g}h(x)|\leq
\delta,
\]
then $h$ has a critical point $x_{0}$ with $d_{g}(x_{0},K)\leq\mu$. Here
$d_{g}$ denotes the geodesic distance associated to the Riemannian metric $g$.

\subsection{The main results for the KGMP-system}

Theorem \ref{thm:main}, together with Propositions \ref{prop:wp} and
\ref{prop:hm}, yields the following results.

\begin{theorem}
\label{cor:wp}Let $\mathfrak{M}$\ be the warped product $M\times_{f^{2}}N$ of
two closed Riemannian manifolds $(M,g)$ and $(N,h)$ with $n:=\dim M=2$ or $3.$
Set $k:=\dim N,$ and let $p\in(2,\infty)$ if $n=2$ and $p\in(2,6)$ if $n=3.$
Assume there exists $\beta\in\mathcal{C}^{1}(M,\mathbb{R})$ such that
$\alpha=\beta\circ\pi_{M}$ and let $K$ be a $\mathcal{C}^{1}$-stable critical
set for the function $\Gamma:=f^{k}\beta^{\frac{p}{p-2}-\frac{n}{2}}$ on $M.$
Then, for $\varepsilon$ small enough, the KGMP-system \emph{(\ref{eq:P-1})}
has a solution $(\mathfrak{u}_{\varepsilon},\mathfrak{v}_{\varepsilon})$ such
that $\mathfrak{u}_{\varepsilon}$ concentrates at the submanifold $\pi
_{M}^{-1}\left(  \xi_{0}\right)  \cong(N,f^{2}(\xi_{0})h)$ for some $\xi
_{0}\in K$ as $\varepsilon\rightarrow0$.
\end{theorem}

\begin{theorem}
\label{cor:hm}Assume there exist a closed Riemannian manifold $M$ with
$n:=\dim M=2$ or $3$ and a harmonic morphism $\pi:\mathfrak{M}\rightarrow M$
whose dilation $\lambda$ is such that $\mu\circ\pi=\lambda^{2}$. Assume
further that $\alpha=\beta\circ\pi$ with $\beta\in\mathcal{C}^{1}%
(M,\mathbb{R})$. Let $p\in(2,\infty)$ if $n=2$ and $p\in(2,6)$ if $n=3,$ and
let $K$ be a $\mathcal{C}^{1}$-stable critical set for the function
$\Gamma:=\beta^{\frac{p}{p-2}-\frac{n}{2}}\mu^{\frac{n}{2}-1}$ on $M$. Then,
for $\varepsilon$ small enough, the KGMP-system \emph{(\ref{eq:P-1})} has a
solution $(\mathfrak{u}_{\varepsilon},\mathfrak{v}_{\varepsilon})$ such that
$\mathfrak{u}_{\varepsilon}$ concentrates at the submanifold $\pi^{-1}\left(
\xi_{0}\right)  $ of $\mathfrak{M}$ for some $\xi_{0}\in K$ as $\varepsilon
\rightarrow0$.
\end{theorem}

This last result applies, in particular, to the standard $3$-sphere
$\mathfrak{M}=\mathbb{S}^{3}$ and the real projective space $\mathfrak{M}%
=\mathbb{R}P^{3}$ for all $p\in(2,\infty)$ with $\mu=\lambda\equiv1,$ see
subsection \ref{subsec:hm}.

The rest of the paper is devoted to the proof of Theorem \ref{thm:main}. In
section \ref{sec:outline} we reduce the system to a single equation and give
the outline of the proof of Theorem \ref{thm:main}, which follows the
well-known Lyapunov-Schmidt reduction procedure. In section \ref{sec:redux} we
establish the Lyapunov-Schmidt reduction and in section \ref{sec:red-exp} we
derive the expansion of the reduced energy functional. Section \ref{sec:tech}%
\ is devoted to the proof of some technical results.

\section{Outline of the proof of Theorem \ref{thm:main}}

\label{sec:outline}

\subsection{Reduction to a single equation}

First, we reduce the system to a single equation. To overcome the problems
caused by the competition between $u$ and $v$, using an idea of Benci and
Fortunato \cite{BF}, we consider the map $\Psi:H_{g}^{1}(M)\rightarrow
H_{g}^{1}(M)$ defined by the equation
\begin{equation}
-\text{div}_{g}\left(  c(x)\nabla_{g}\Psi(u)\right)  +b(x)(1+q^{2}u^{2}%
)\Psi(u)=b(x)qu^{2}. \label{eq:e1}%
\end{equation}
It follows from standard variational arguments that $\Psi$ is well-defined in
$H_{g}^{1}(M)$.

Using the maximum principle and regularity theory it is not hard to prove
that
\begin{equation}
0<\Psi(u)<1/q\text{\qquad for all }u\in H_{g}^{1}(M).\label{psipos}%
\end{equation}
For the proofs of the following two lemmas we refer to \cite{DH}.

\begin{lemma}
\label{lem:e1}The map $\Psi:H_{g}^{1}(M)\rightarrow H_{g}^{1}(M)$ is of class
$\mathcal{C}^{1},$ and its differential $V_{u}:=\Psi^{\prime}(u)$ at $u$ is
defined by
\begin{equation}
-\text{\emph{div}}_{g}\left(  c(x)\nabla_{g}V_{u}[h]\right)  +b(x)\left(
1+q^{2}u^{2}\right)  V_{u}[h]=2b(x)qu(1-q\Psi(u))h \label{eq:e2}%
\end{equation}
for every $h\in H_{g}^{1}(M).$ Moreover,
\[
0\leq\Psi^{\prime}(u)[u]\leq\frac{2}{q}\text{\qquad for all }u\in H_{g}%
^{1}(M).
\]

\end{lemma}

\begin{lemma}
\label{lem:e2}The map $\Theta:H_{g}^{1}(M)\rightarrow\mathbb{R}$ given by
\[
\Theta(u):=\frac{1}{2}\int_{M}b(x)(1-q\Psi(u))u^{2}d\mu_{g}%
\]
is of class $\mathcal{C}^{1}$ and
\[
\Theta^{\prime}(u)[h]=\int_{M}b(x)(1-q\Psi(u))^{2}uh\,d\mu_{g}\text{\qquad for
all }u,h\in H_{g}^{1}(M).
\]

\end{lemma}

Next, we introduce the functionals $I_{\varepsilon},J_{\varepsilon
},G_{\varepsilon}:H_{g}^{1}(M)\rightarrow\mathbb{R}$ given by%
\begin{equation}
I_{\varepsilon}(u):=J_{\varepsilon}(u)+\frac{\omega^{2}}{2}G_{\varepsilon
}(u),\label{ieps}%
\end{equation}
where
\[
J_{\varepsilon}(u):=\frac{1}{2\varepsilon^{2}}\int\limits_{M}\left[
\varepsilon^{2}c(x)|\nabla_{g}u|^{2}+d(x)u^{2}\right]  d\mu_{g}-\frac
{1}{p\varepsilon^{2}}\int\limits_{M}b(x)\left(  u^{+}\right)  ^{p}d\mu_{g}%
\]
with $d(x):=a(x)-\omega^{2}b(x)$, and
\[
G_{\varepsilon}(u):=\frac{q}{\varepsilon^{2}}\int_{M}b(x)\Psi(u)u^{2}d\mu_{g}.
\]
From Lemma \ref{lem:e2} we deduce that
\[
\frac{1}{2}G_{\varepsilon}^{\prime}(u)[\varphi]=\frac{1}{\varepsilon^{2}}%
\int_{M}b(x)[2q\Psi(u)-q^{2}\Psi^{2}(u)]u\varphi\,d\mu_{g}.
\]
Hence,
\[
I_{\varepsilon}^{\prime}(u)\varphi=\frac{1}{\varepsilon^{2}}\int
_{M}\varepsilon^{2}c(x)\nabla_{g}u\nabla_{g}\varphi+a(x)u\varphi
-b(x)(u^{+})^{p-1}\varphi-b(x)\omega^{2}(1-q\Psi(u))^{2}u\varphi\,d\mu_{g}.
\]

Therefore, if $u$ is a critical point of the functional $I_{\varepsilon},$
then $u$ solves the problem%
\begin{equation}
\left\{
\begin{array}
[c]{l}%
-\varepsilon^{2}\text{div}_{g}\left(  c(x)\nabla_{g}u\right)  +(a(x)-\omega
^{2}b(x))u+\omega^{2}qb(x)\Psi(u)(2-q\Psi(u))u=b(x)(u^{+})^{p-1},\\
u\in H_{g}^{1}(M).
\end{array}
\right.  \label{p}%
\end{equation}
If $u\neq0$ by the maximum principle and regularity theory we have that $u>0$.
Thus the pair $(u,\Psi(u))$ is a solution of the system (\ref{eq:P}). This
reduces the existence problem for the system (\ref{eq:P}) to showing that the
functional $I_{\varepsilon}$ has a nontrivial critical point.

\subsection{The limit problems}

Theorem \ref{thm:main}\ concerns manifolds of dimensions $2$ and $3.$ To
simplify the exposition we shall treat in full detail only the case $n=2$.
Everything can be extended in a straightforward way to the case $n=3,$ except
for the estimates in section \ref{sec:tech}. These estimates, however, were
computed in the appendix of \cite{GMP} for $n=3$.

Henceforth, we assume that $\dim M=2.$ We fix $r>0$ smaller than the
injectivity radius of $M.$ We identify the tangent space of $M$ at $\xi$ with
$\mathbb{R}^{2}$ and denote by $B(x,r)$ the ball in $\mathbb{R}^{2}$ centered
at $x$ of radius $r$ and by $B_{g}(\xi,r)$ the ball in $M$ centered at $\xi$
of radius $r$, with respect to the distance induced by the Riemannian metric
$g.$ The exponential map $\exp_{\xi}:B(0,r)\rightarrow B_{g}(\xi,r)$ provides
local coordenates on $M,$ which are called normal coordinates. We denote by
$g_{\xi}$ the Riemannian metric at $\xi$ given in normal coordinates by the
matrix $\left(  g_{ij}\right)  .$ We denote the inverse matrix by
$(g^{ij}(z)):=\left(  g_{ij}(z)\right)  ^{-1}$ and write $\left\vert g_{\xi
}(z)\right\vert :=\det\left(  g_{ij}(z)\right)  .$ Then, we have that
\begin{align}
g^{ij}(\varepsilon z)  &  =\delta_{ij}+\frac{\varepsilon^{2}}{2}\sum
_{r,k=1}^{n}\frac{\partial^{2}g^{ij}}{\partial z_{r}\partial z_{k}}%
(0)z_{r}z_{k}+O(\varepsilon^{3}|z|^{3})=\delta_{ij}+o(\varepsilon
),\label{eq:espg1}\\
\left\vert g(\varepsilon z)\right\vert ^{\frac{1}{2}}  &  =1-\frac
{\varepsilon^{2}}{4}\sum_{i,r,k=1}^{n}\frac{\partial^{2}g^{ii}}{\partial
z_{r}\partial z_{k}}(0)z_{r}z_{k}+O(\varepsilon^{3}|z|^{3})=1+o(\varepsilon).
\label{eq:espg2}%
\end{align}
Here $\delta_{ij}$ denotes the Kronecker symbol.

For $p\in(2,\infty)$ and $\xi\in M,$ set%
\[
{A(\xi):=\frac{a(\xi)}{c(\xi)},}\text{\qquad}{B(\xi):=\frac{b(\xi)}{c(\xi
)},\qquad,\gamma(\xi):=}\left(  {\frac{a(\xi)}{b(\xi)}}\right)  ^{\frac
{1}{p-2}}{.}%
\]
We consider the problem
\[
-c(\xi)\Delta V+a(\xi)V=b(\xi)V^{p-1},\text{\qquad}V\in H^{1}(\mathbb{R}^{2}),
\]
and denote by $V^{\xi}$ its unique positive spherically symmetric solution.
This problem is equivalent to
\[
-\Delta V+A(\xi)V=B(\xi)V^{p-1},\text{\qquad}V\in H^{1}(\mathbb{R}^{2}).
\]
The function $V^{\xi}$ and its derivatives decay exponentially at infinity.
$V^{\xi}$ can be written as
\[
V^{\xi}(z)=\gamma(\xi)U(\sqrt{A(\xi)}z),
\]
where $U$ is the unique positive spherically symmetric solution to
\[
-\Delta U+U=U^{p-1},\text{\qquad}U\in H^{1}(\mathbb{R}^{2}).
\]

For $\xi\in M$ and $\varepsilon>0$ we define $W_{\varepsilon,\xi}\in H_{g}%
^{1}(M)$ by
\[
W_{\varepsilon,\xi}(x):=\left\{
\begin{array}
[c]{ll}%
V^{\xi}\left(  \frac{1}{\varepsilon}\exp_{\xi}^{-1}(x)\right)  \chi\left(
\exp_{\xi}^{-1}(x)\right)  & \text{if }x\in B_{g}(\xi,r),\\
0 & \text{otherwise,}%
\end{array}
\right.
\]
where $\chi\in\mathcal{C}^{\infty}(\mathbb{R}^{n})$ is a radial cut-off
function such that $\chi(z)=1$ if $\left\vert z\right\vert \leq r/2$ and
$\chi(z)=0$ if $\left\vert z\right\vert \geq r.$ Setting $V_{\varepsilon
}(z):=V\left(  \frac{z}{\varepsilon}\right)  $ and $y:=\exp_{\xi}^{-1}x$ we
have that
\[
W_{\varepsilon,\xi}(\exp_{\xi}(y))=V^{\xi}\left(  \frac{y}{\varepsilon
}\right)  \chi(y)=V_{\varepsilon}^{\xi}(y)\chi(y),
\]
so the function $W_{\varepsilon,\xi}$ is simply the function $V^{\xi}$
rescaled, cut off and read in normal coordinates at $\xi$ in $M.$

Similarly, for $i=1,2$ we define
\[
Z_{\varepsilon,\xi}^{i}(x)=\left\{
\begin{array}
[c]{ll}%
\psi_{\xi}^{i}\left(  \frac{1}{\varepsilon}\exp_{\xi}^{-1}(x)\right)
\chi\left(  \exp_{\xi}^{-1}(x)\right)   & \text{if }x\in B_{g}(\xi,r),\\
0 & \text{otherwise,}%
\end{array}
\right.
\]
where
\[
\psi_{\xi}^{i}(\eta)=\frac{\partial}{\partial\eta_{i}}V^{\xi}(\eta)=\gamma
(\xi)\sqrt{A(\xi)}\frac{\partial U}{\partial\eta_{i}}(\sqrt{A(\xi)}\eta).
\]
The functions $\psi_{\xi}^{i}$ are solutions of the linearized equation
\[
-\Delta\psi+A(\xi)\psi=(p-1)B(\xi)\left(  V^{\xi}\right)  ^{p-2}\psi\text{
\ \ in }\mathbb{R}^{2}.
\]

\begin{proposition}
\label{prop:Zi}There is a positive constant $C$ such that
\[
\left\langle Z_{\varepsilon,\xi}^{h},Z_{\varepsilon,\xi}^{k}\right\rangle
_{\varepsilon}=C\delta_{hk}+o(1),
\]
as $\varepsilon\rightarrow0.$
\end{proposition}

\begin{proof}
From the Taylor expansions of $g^{ij}(\varepsilon z)$, $|g(\varepsilon
z)|^{\frac{1}{2}}$, $a(\exp_{\xi}(\varepsilon z))$ and $c(\exp_{\xi
}(\varepsilon z))$ we obtain
\begin{align*}
&  \left\langle Z_{\varepsilon,\xi}^{h},Z_{\varepsilon,\xi}^{k}\right\rangle
_{\varepsilon}=\frac{1}{\varepsilon^{2}}\int_{M}\varepsilon^{2}c(x)\nabla
_{g}Z_{\varepsilon,\xi}^{h}(x)\nabla_{g}Z_{\varepsilon,\xi}^{k}%
(x)+d(x)Z_{\varepsilon,\xi}^{h}(x)Z_{\varepsilon,\xi}^{k}(x)d\mu_{g}\\
&  =\int_{B(0,r/\varepsilon)}\sum_{ij}c(\exp_{\xi}(\varepsilon z))g_{\xi}%
^{ij}(\varepsilon z)\frac{\partial}{\partial z_{i}}(\psi_{\xi}^{h}%
(z)\chi(\varepsilon z))\frac{\partial}{\partial z_{j}}(\psi_{\xi}^{h}%
(z)\chi(\varepsilon z))|g_{\xi}(\varepsilon z)|^{\frac{1}{2}}dz\\
&  \qquad+\int_{B(0,r/\varepsilon)}d(\exp_{\xi}(\varepsilon z))\psi_{\xi}%
^{h}(z)\psi_{\xi}^{h}(z)\chi^{2}(\varepsilon z)|g_{\xi}(\varepsilon
z)|^{\frac{1}{2}}dz\\
&  =c(\xi)\int_{\mathbb{R}^{2}}\nabla\psi_{\xi}^{h}\nabla\psi_{\xi}%
^{h}dz+d(\xi)\int_{\mathbb{R}^{2}}\psi_{\xi}^{h}\psi_{\xi}^{k}dz+o(1)=C\delta
_{hk}+o(1),
\end{align*}
as claimed.
\end{proof}

Next, we compute the derivatives of $W_{\varepsilon,\xi}$ with respect to
$\xi$ in normal coordinates. Fix $\xi_{0}\in M$. We write the points $\xi\in
B_{g}(\xi_{0},r)$ as
\[
\xi=\xi(y)=\exp_{\xi_{0}}(y)\text{\qquad with }y\in B(0,r).
\]
We define
\[
\mathcal{E}(y,x)=\exp_{\xi(y)}^{-1}(x)=\exp_{\exp_{\xi_{0}}(y)}^{-1}(x),
\]
where $x\in B_{g}(\xi(y),r)$ and $y\in B(0,r)$. Then we can write
\begin{align*}
W_{\varepsilon,\xi(y)}(x) &  =\gamma(\xi(y))U_{\varepsilon}(\sqrt{A(\xi
(y))}\exp_{\xi(y)}^{-1}(x))\chi(\exp_{\xi(y)}^{-1}(x))\\
&  =\tilde{\gamma}(y)U_{\varepsilon}(\sqrt{\tilde{A}(y)}\mathcal{E}%
(y,x))\chi(\mathcal{E}(y,x))
\end{align*}
where $\tilde{A}(y)=A(\exp_{\xi_{0}}(y))$ and $\tilde{\gamma}(y)=\gamma
(\exp_{\xi_{0}}(y))$. Thus we have
\begin{align*}
\left.  \frac{\partial}{\partial y_{s}}W_{\varepsilon,\xi(y)}\right\vert
_{y=0} &  =\left(  \left.  \frac{\partial}{\partial y_{s}}\tilde{\gamma
}(y)\right\vert _{y=0}\right)  U\left(  \frac{1}{\varepsilon}\sqrt{\tilde
{A}(0)}\mathcal{E}(0,x)\right)  \chi(\mathcal{E}(0,x))\\
&  +\tilde{\gamma}(0)U\left(  \frac{1}{\varepsilon}\sqrt{\tilde{A}%
(0)}\mathcal{E}(0,x)\right)  \left.  \frac{\partial}{\partial y_{s}}%
\chi\left(  \mathcal{E}_{k}(y,x)\right)  \right\vert _{y=0}\\
&  +\tilde{\gamma}(0)\chi(\mathcal{E}(0,x))\left.  \frac{\partial}{\partial
y_{s}}U\left(  \frac{1}{\varepsilon}\sqrt{\tilde{A}(y)}\mathcal{E}%
(y,x)\right)  \right\vert _{y=0}.
\end{align*}
If $x=\exp_{\xi_{0}}\varepsilon z$, $\xi_{0}=\xi(0)$, then $\mathcal{E}%
(0,x)=\varepsilon z$ and we have
\begin{align}
\left.  \frac{\partial}{\partial y_{s}}W_{\varepsilon,\xi(y)}\right\vert
_{y=0} &  =\left(  \left.  \frac{\partial}{\partial y_{s}}\tilde{\gamma
}(y)\right\vert _{y=0}\right)  U(\sqrt{\tilde{A}(0)}z)\chi(\varepsilon
z)\nonumber\\
&  +\tilde{\gamma}(0)U\left(  \sqrt{\tilde{A}(0)}z\right)  \frac{\partial\chi
}{\partial\eta_{k}}(\varepsilon z)\left.  \frac{\partial}{\partial y_{s}%
}\mathcal{E}_{k}(y,\exp_{\xi_{0}}\varepsilon z)\right\vert _{y=0}%
\label{eq:derWeps-1}\\
&  +\tilde{\gamma}(0)\chi(\varepsilon z)\frac{\sqrt{\tilde{A}(0)}}%
{\varepsilon}\frac{\partial U}{\partial\eta_{k}}\left(  \sqrt{\tilde{A}%
(0)}z\right)  \left.  \frac{\partial}{\partial y_{s}}\mathcal{E}_{k}%
(y,\exp_{\xi_{0}}\varepsilon z)\right\vert _{y=0}.\nonumber
\end{align}
We also recall the following Taylor expansions:
\begin{equation}
\frac{\partial}{\partial y_{h}}\mathcal{E}_{k}(0,\exp_{\xi_{0}}\varepsilon
z)=-\delta_{hk}+O(\varepsilon^{2}|z|^{2}).\label{eq:espE}%
\end{equation}

\subsection{Outline of the proof of Theorem \ref{thm:main}}

Let $H_{\varepsilon}$ denote the Hilbert space $H_{g}^{1}(M)$ equipped with
the inner product
\[
\left\langle u,v\right\rangle _{\varepsilon}:=\frac{1}{\varepsilon^{2}}\left(
\varepsilon^{2}\int_{M}c(x)\nabla_{g}u\nabla_{g}v\,d\mu_{g}+\int
_{M}d(x)uv\,d\mu_{g}\right)  ,
\]
which induces the norm
\[
\Vert u\Vert_{\varepsilon}^{2}:=\frac{1}{\varepsilon^{2}}\left(
\varepsilon^{2}\int_{M}c(x)|\nabla_{g}u|^{2}d\mu_{g}+\int_{M}d(x)u^{2}d\mu
_{g}\right)  ,
\]
with $d(x):=a(x)-\omega^{2}b(x)>0$. Similarly, let $L_{\varepsilon}^{q}$ be
the Banach space $L_{g}^{q}(M)$ with the norm
\[
|u|_{q,\varepsilon}:=\left(  \frac{1}{\varepsilon^{2}}\int_{M}|u|^{q}d\mu
_{g}\right)  ^{1/q}.
\]

Since we are assuming that $\dim M=2,$ for each $q\geq2$ the embedding
$H_{\varepsilon}\hookrightarrow L_{\varepsilon}^{q}$ is continuous. In fact,
there is a positive constant $C$, independent of $\varepsilon,$ such that%
\begin{equation}
|u|_{q,\varepsilon}\leq C\left\Vert u\right\Vert _{\varepsilon}\qquad\forall
u\in H_{\varepsilon}, \label{r2}%
\end{equation}
Moreover, this embedding is compact.

Fix $p\in(2,\infty).$ The adjoint operator $i_{\varepsilon}^{\ast
}:L_{\varepsilon}^{p^{\prime}}\rightarrow H_{\varepsilon},$ $p^{\prime}%
:=\frac{p}{p-1},$ to the embedding $i_{\varepsilon}:H_{\varepsilon
}\hookrightarrow L_{\varepsilon}^{p}$ is defined by
\begin{align*}
u=i_{\varepsilon}^{\ast}(v)\Leftrightarrow &  \left\langle u,\varphi
\right\rangle _{\varepsilon}=\frac{1}{\varepsilon^{2}}\int_{M}v\varphi
\quad\forall\varphi\in H_{\varepsilon}\\
\Leftrightarrow &  -\varepsilon^{2}\text{div}_{g}\left(  c(x)\nabla
_{g}u\right)  +d(x)u=v,\quad u\in H_{g}^{1}(M).
\end{align*}
One has that
\begin{equation}
\left\Vert i_{\varepsilon}^{\ast}(v)\right\Vert _{\varepsilon}\leq
C|v|_{p^{\prime},\varepsilon}\qquad\forall v\in L_{\varepsilon}^{p^{\prime}},
\label{eq:istar}%
\end{equation}
where the constant $C$ does not depend on $\varepsilon.$

Using the adjoint operator we can rewrite problem (\ref{p}) as
\begin{equation}
u=i_{\varepsilon}^{\ast}\left[  b(x)f(u)+\omega^{2}b(x)g(u)\right]  ,\qquad
u\in H_{\varepsilon}, \label{p1}%
\end{equation}
where
\[
f(u):=\left(  u^{+}\right)  ^{p-1}\text{\qquad and\qquad}g(u):=\left(
q^{2}\Psi^{2}(u)-2q\Psi(u)\right)  u.
\]

Let%
\[
K_{\varepsilon,\xi}:=\text{Span }\left\{  Z_{\varepsilon,\xi}^{1}%
,Z_{\varepsilon,\xi}^{2}\right\}
\]
and%
\[
K_{\varepsilon,\xi}^{\bot}:=\left\{  \phi\in H_{\varepsilon}:\left\langle
\phi,Z_{\varepsilon,\xi}^{i}\right\rangle _{\varepsilon}=0,\ i=1,2\right\}  .
\]
We denote the projections onto these subspaces by
\[
\Pi_{\varepsilon,\xi}:H_{\varepsilon}\rightarrow K_{\varepsilon,\xi
}\text{\qquad and\qquad}\Pi_{\varepsilon,\xi}^{\bot}:H_{\varepsilon
}\rightarrow K_{\varepsilon,\xi}^{\bot}.
\]

We look for a solution of (\ref{p}) of the form%
\[
u_{\varepsilon}:=W_{\varepsilon,\xi}+\phi\text{\qquad with \ }\phi\in
K_{\varepsilon,\xi}^{\bot}.
\]
This is equivalent to solving the pair of equations
\begin{align}
\Pi_{\varepsilon,\xi}^{\perp}\left\{  W_{\varepsilon,\xi}+\phi-i_{\varepsilon
}^{\ast}\left[  b(x)f\left(  W_{\varepsilon,\xi}+\phi\right)  +\omega
^{2}b(x)g\left(  W_{\varepsilon,\xi}+\phi\right)  \right]  \right\}   &
=0,\label{red1}\\
\Pi_{\varepsilon,\xi}\left\{  W_{\varepsilon,\xi}+\phi-i_{\varepsilon}^{\ast
}\left[  b(x)f\left(  W_{\varepsilon,\xi}+\phi\right)  +\omega^{2}b(x)g\left(
W_{\varepsilon,\xi}+\phi\right)  \right]  \right\}   &  =0. \label{red2}%
\end{align}

The first step of the proof of Theorem \ref{thm:main} is to solve equation
(\ref{red1}). More precisely, for any fixed $\xi\in M$ and $\varepsilon$ small
enough, we will show that there is a function $\phi\in K_{\varepsilon,\xi
}^{\perp}$ such that (\ref{red1}) holds. To do this we consider the linear
operator $L_{\varepsilon,\xi}:K_{\varepsilon,\xi}^{\perp}\rightarrow
K_{\varepsilon,\xi}^{\perp}$ given by
\[
L_{\varepsilon,\xi}(\phi):=\Pi_{\varepsilon,\xi}^{\perp}\left\{
\phi-i_{\varepsilon}^{\ast}\left[  b(x)f^{\prime}\left(  W_{\varepsilon,\xi
}\right)  \phi\right]  \right\}  .
\]
For the proof of the following statement we refer to Lemma 4.1 of \cite{CGM}
(see also Proposition 3.1 of \cite{MP}).

\begin{proposition}
\label{inv} There exist $\varepsilon_{0}>0$ and $C>0$ such that, for every
$\varepsilon\in(0,\varepsilon_{0}),$ $\xi\in M$ and $\phi\in K_{\varepsilon
,\xi}^{\perp},$
\[
\left\Vert L_{\varepsilon,\xi}(\phi)\right\Vert _{\varepsilon}\geq C\Vert
\phi\Vert_{\varepsilon}.
\]

\end{proposition}

This result allows to use a contraction mapping argument to solve equation
(\ref{red1}). The following statement is proved in section \ref{sec:redux}.

\begin{proposition}
\label{resto} There exist $\varepsilon_{0}>0$ and $C>0$ such that, for each
$\xi\in M$ and each $\varepsilon\in(0,\varepsilon_{0}),$ there exists a unique
$\phi_{\varepsilon,\xi}\in K_{\varepsilon,\xi}^{\perp}$ which solves equation
\emph{(\ref{red1}).} Moreover,
\[
\left\Vert \phi_{\varepsilon,\xi}\right\Vert _{\varepsilon}\leq C\varepsilon.
\]
The map $\xi\mapsto\phi_{\varepsilon,\xi}$ is a $\mathcal{C}^{1}$-map.
\end{proposition}

The second step is to solve equation (\ref{red2}). More precisely, for
$\varepsilon$ small enough we will find a point $\xi$ in $M$ such that
equation (\ref{red2}) is satisfied. To this end we introduce the reduced
energy function $\widetilde{I}_{\varepsilon}:M\rightarrow\mathbb{R}$ defined
by
\[
\widetilde{I}_{\varepsilon}(\xi):=I_{\varepsilon}\left(  W_{\varepsilon,\xi
}+\phi_{\varepsilon,\xi}\right)  ,
\]
where $I_{\varepsilon}$ is the variational functional defined in (\ref{ieps})
whose critical points are the solutions to problem (\ref{p}). It is easy to
verify that $\xi_{\varepsilon}$ is a critical point of $\widetilde
{I}_{\varepsilon}$ if and only if the function $u_{\varepsilon}=W_{\varepsilon
,\xi_{\varepsilon}}+\phi_{\varepsilon,\xi_{\varepsilon}}$ is a critical point
of $I_{\varepsilon}$.

In Lemmas \ref{fine4} and \ref{lem:expJeps} we compute the asymptotic
expansion of the reduced functional $\tilde{I}_{\varepsilon}$ with respect to
the parameter $\varepsilon$. We prove the following result.

\begin{proposition}
\label{lem:tool2}The expansion
\[
\tilde{I}_{\varepsilon}(\xi)=C\frac{c(\xi)^{\frac{n}{2}}a(\xi)^{\frac{p}%
{p-2}-\frac{n}{2}}}{b(\xi)^{\frac{2}{p-2}}}+o(1)=C\Gamma(\xi)+o(1),
\]
holds true $\mathcal{C}^{1}$-uniformly with respect to $\xi$ as $\varepsilon
\rightarrow0$, where $C=\left(  \frac{1}{2}-\frac{1}{p}\right)  \int
_{\mathbb{R}^{n}}U^{p}dz$.
\end{proposition}

Using the previous propositions we now prove Theorem \ref{thm:main}.

\begin{proof}
[\textbf{Proof of Theorem \ref{thm:main}.}]Since $K$ is a $\mathcal{C}^{1}%
$-stable critical set for $\Gamma,$ by Proposition \ref{lem:tool2} $\tilde
{I}_{\varepsilon}$ has a critical point $\xi_{\varepsilon}\in M$ such that
$d_{g}(\xi_{\varepsilon},K)\rightarrow0$ as $\varepsilon\rightarrow0$. Hence,
$u_{\varepsilon}=W_{\varepsilon,\xi_{\varepsilon}}+\phi_{\varepsilon
,\xi_{\varepsilon}}$ is a solution of (\ref{p}), and the pair $(u_{\varepsilon
},\Psi(u_{\varepsilon}))$ is a solution to the system (\ref{eq:P}) such that
$u_{\varepsilon}$ concentrates at a point $\xi_{0}\in K$ as $\varepsilon
\rightarrow0.$
\end{proof}

\section{The finite dimensional reduction}

\label{sec:redux}This section is devoted to the proof of Proposition
\ref{resto}. We denote by
\begin{equation}
\Vert u\Vert_{g}^{2}:=\int_{M}\left(  |\nabla_{g}u|^{2}+u^{2}\right)  d\mu
_{g}\qquad\text{and}\qquad|u|_{g,q}^{q}:=\int_{M}|u|^{q}d\mu_{g}\label{norms}%
\end{equation}
the standard norms in the spaces $H_{g}^{1}(M)$ and $L^{q}(M).$

Equation (\ref{red1}) is equivalent to
\begin{equation}
L_{\varepsilon,\xi}(\phi)=N_{\varepsilon,\xi}(\phi)+S_{\varepsilon,\xi}%
(\phi)+R_{\varepsilon,\xi},\label{i1}%
\end{equation}
where
\[
N_{\varepsilon,\xi}(\phi):=\Pi_{\varepsilon,\xi}^{\perp}\left\{
i_{\varepsilon}^{\ast}\left[  b(x)\left(  f\left(  W_{\varepsilon,\xi}%
+\phi\right)  -f\left(  W_{\varepsilon,\xi}\right)  -f^{\prime}\left(
W_{\varepsilon,\xi}\right)  \right)  \phi\right]  \right\}  ,
\]%
\[
S_{\varepsilon,\xi}(\phi):=\omega^{2}\Pi_{\varepsilon,\xi}^{\perp}\left\{
i_{\varepsilon}^{\ast}\left[  b(x)\left(  q^{2}\Psi^{2}\left(  W_{\varepsilon
,\xi}+\phi\right)  -2q\Psi\left(  W_{\varepsilon,\xi}+\phi\right)  \right)
\left(  W_{\varepsilon,\xi}+\phi\right)  \right]  \right\}  ,
\]%
\[
R_{\varepsilon,\xi}:=\Pi_{\varepsilon,\xi}^{\perp}\left\{  i_{\varepsilon
}^{\ast}\left[  b(x)f\left(  W_{\varepsilon,\xi}\right)  \right]
-W_{\varepsilon,\xi}\right\}  .
\]
In order to solve equation (\ref{i1}) we will show that the operator
$T_{\varepsilon,\xi}:K_{\varepsilon,\xi}^{\perp}\rightarrow K_{\varepsilon
,\xi}^{\perp}$ defined by
\[
T_{\varepsilon,\xi}(\phi):=L_{\varepsilon,\xi}^{-1}\left(  N_{\varepsilon,\xi
}(\phi)+S_{\varepsilon,\xi}(\phi)+R_{\varepsilon,\xi}\right)
\]
has a fixed point. To this end we prove that $T_{\varepsilon,\xi}$ is a
contraction mapping on suitable ball in $H_{\varepsilon}.$ We start with an
estimate for $R_{\varepsilon,\xi}$.

\begin{lemma}
\label{re1}There exist $\varepsilon_{0}>0$ and $C>0$ such that, for any
$\xi\in M$ and any $\varepsilon\in(0,\varepsilon_{0}),$ the inequality
\[
\left\Vert R_{\varepsilon,\xi}\right\Vert _{\varepsilon}\leq C\varepsilon
\]
holds true.
\end{lemma}

\begin{proof}
See Lemma 4.2 in \cite{CGM}.
\end{proof}

Next, we give an estimate for $N_{\varepsilon,\xi}(\phi)$.

\begin{lemma}
\label{ne1}There exist $\varepsilon_{0}>0$, $C>0$ and $\widetilde{C}\in(0,1)$
such that, for any $\xi\in M,$ $\varepsilon\in(0,\varepsilon_{0})$ and $R>0,$
the inequalities
\begin{equation}
\Vert N_{\varepsilon,\xi}(\phi)\Vert_{\varepsilon}\leq C(\Vert\phi
\Vert_{\varepsilon}^{2}+\Vert\phi\Vert_{\varepsilon}^{p-1}),\label{eq:N1}%
\end{equation}
\begin{equation}
\Vert N_{\varepsilon,\xi}(\phi_{1})-N_{\varepsilon,\xi}(\phi_{2}%
)\Vert_{\varepsilon}\leq\widetilde{C}\Vert\phi_{1}-\phi_{2}\Vert_{\varepsilon
},\label{eq:N2}%
\end{equation}
hold true for $\phi,\phi_{1},\phi_{2}\in\left\{  \phi\in H_{\varepsilon}%
:\Vert\phi\Vert_{\varepsilon}\leq R\varepsilon\right\}  .$
\end{lemma}

\begin{proof}
By direct computation we obtain
\begin{equation}
|f^{\prime}(W_{\varepsilon,\xi}+v)-f^{\prime}(W_{\varepsilon,\xi}%
)|\leq\left\{
\begin{array}
[c]{ll}%
CW_{\varepsilon,\xi}^{p-3}|v| & 2<p<3,\\
C(W_{\varepsilon,\xi}^{p-3}|v|+|v|^{p-2}) & p\geq3.
\end{array}
\right.  \label{eq:stimaf}%
\end{equation}
From the mean value theorem and inequality (\ref{eq:istar}) we derive
\[
\Vert N_{\varepsilon,\xi}(\phi_{1})-N_{\varepsilon,\xi}(\phi_{2}%
)\Vert_{\varepsilon}\leq C\left\vert f^{\prime}(W_{\varepsilon,\xi}+\phi
_{2}+t(\phi_{1}-\phi_{2}))-f^{\prime}(W_{\varepsilon,\xi})\right\vert
_{\frac{p}{p-2},\varepsilon}\Vert\phi_{1}-\phi_{2}\Vert_{\varepsilon}.
\]
Using (\ref{eq:stimaf}) we conclude that
\[
C\left\vert f^{\prime}(W_{\varepsilon,\xi}+\phi_{2}+t(\phi_{1}-\phi
_{2}))-f^{\prime}(W_{\varepsilon,\xi})\right\vert _{\frac{p}{p-2},\varepsilon
}<1
\]
provided $\Vert\phi_{1}\Vert_{\varepsilon}$ and $\Vert\phi_{2}\Vert
_{\varepsilon}$ are small enough. The same estimates yield (\ref{eq:N1}).
\end{proof}

Now we estimate $S_{\varepsilon,\xi}(\phi)$.

\begin{lemma}
\label{se1}There exists $\varepsilon_{0}>0$ and $C>0$ such that, for any
$\xi\in M$, $\varepsilon\in(0,\varepsilon_{0})$ and $R>0,$ the inequalities
\begin{equation}
\Vert S_{\varepsilon,\xi}(\phi)\Vert_{\varepsilon}\leq C\varepsilon,
\label{se11}%
\end{equation}
\begin{equation}
\Vert S_{\varepsilon,\xi}(\phi_{1})-S_{\varepsilon,\xi}(\phi_{2}%
)\Vert_{\varepsilon}\leq\ell_{\varepsilon}\Vert\phi_{1}-\phi_{2}%
\Vert_{\varepsilon}, \label{se12}%
\end{equation}
hold true for $\phi,\phi_{1},\phi_{2}\in\left\{  \phi\in H_{\varepsilon}%
:\Vert\phi\Vert_{\varepsilon}\leq R\varepsilon\right\}  ,$ where
$\ell_{\varepsilon}\rightarrow0$ as $\varepsilon\rightarrow0$.
\end{lemma}

\begin{proof}
Let us prove \eqref{se11}. From the definition of $i^{\ast}$ and inequality
(\ref{eq:istar}) we derive
\begin{align*}
\Vert S_{\varepsilon,\xi}(\phi)\Vert_{\varepsilon} &  \leq C\left(  \left\vert
\Psi^{2}\left(  W_{\varepsilon,\xi}+\phi\right)  \left(  W_{\varepsilon,\xi
}+\phi\right)  \right\vert _{p^{\prime},\varepsilon}+\left\vert \Psi\left(
W_{\varepsilon,\xi}+\phi\right)  \left(  W_{\varepsilon,\xi}+\phi\right)
\right\vert _{p^{\prime},\varepsilon}\right)  \\
&  =:I_{1}+I_{2}.
\end{align*}
For any $t\in\left(  2,\infty\right)  ,$ setting $s:=\frac{tp^{\prime}%
}{t-p^{\prime}}$ and $\vartheta:=\frac{2}{t^{\prime}}\in(1,2)$ and applying
Lemma \ref{lem:e3} and Remark \ref{remark:Weps}, we obtain
\begin{align*}
I_{2} &  \leq C\frac{1}{\varepsilon^{2/p^{\prime}}}\left(  \int_{M}\left\vert
\Psi\left(  W_{\varepsilon,\xi}+\phi\right)  \right\vert ^{t}d\mu_{g}\right)
^{\frac{1}{t}}\left(  \int_{M}\left\vert W_{\varepsilon,\xi}+\phi\right\vert
^{s}d\mu_{g}\right)  ^{\frac{1}{s}}\\
&  \leq C\frac{1}{\varepsilon^{2/p^{\prime}}}\left\Vert \Psi\left(
W_{\varepsilon,\xi}+\phi\right)  \right\Vert _{g}\left(  \varepsilon^{\frac
{2}{s}}\left(  \frac{1}{\varepsilon^{2}}\int_{M}\left\vert W_{\varepsilon,\xi
}\right\vert ^{s}d\mu_{g}\right)  ^{\frac{1}{s}}+\left\vert \phi\right\vert
_{g,s}\right)  \\
&  \leq C\frac{1}{\varepsilon^{2/p^{\prime}}}\left(  \varepsilon^{\vartheta
}+\Vert\phi\Vert_{\varepsilon}^{2}\right)  \left(  \varepsilon^{\frac{2}{s}%
}+\Vert\phi\Vert_{\varepsilon}\right)  \\
&  \leq C\left(  \varepsilon^{\vartheta+\frac{2}{s}-\frac{2}{p^{\prime}}%
}+\varepsilon^{\vartheta+1-\frac{2}{p^{\prime}}}\right)  =C\left(
\varepsilon^{\vartheta-\frac{2}{t}}+\varepsilon^{\vartheta+1-\frac
{2}{p^{\prime}}}\right)  \\
&  \leq C\varepsilon
\end{align*}
for all $\Vert\phi\Vert_{\varepsilon}\leq R\varepsilon$. From this estimate we
deduce that $I_{1}\leq C\varepsilon$ and, hence, \eqref{se11} follows.

Next, we prove \eqref{se12}. From inequality (\ref{eq:istar}) we obtain that
\begin{align*}
\Vert S_{\varepsilon,\xi}(\phi_{1})-S_{\varepsilon,\xi}(\phi_{2}%
)\Vert_{\varepsilon}\leq &  C\left\vert \left[  \Psi\left(  W_{\varepsilon
,\xi}+\phi_{1}\right)  -\Psi\left(  W_{\varepsilon,\xi}+\phi_{2}\right)
\right]  W_{\varepsilon,\xi}\right\vert _{p^{\prime},\varepsilon}\\
&  +C\left\vert \left[  \Psi^{2}\left(  W_{\varepsilon,\xi}+\phi_{1}\right)
-\Psi^{2}\left(  W_{\varepsilon,\xi}+\phi_{2}\right)  \right]  W_{\varepsilon
,\xi}\right\vert _{p^{\prime},\varepsilon}\\
&  +C\left\vert \Psi\left(  W_{\varepsilon,\xi}+\phi_{1}\right)  \phi_{1}%
-\Psi\left(  W_{\varepsilon,\xi}+\phi_{2}\right)  \phi_{2}\right\vert
_{p^{\prime},\varepsilon}\\
&  +C\left\vert \Psi^{2}\left(  W_{\varepsilon,\xi}+\phi_{1}\right)  \phi
_{1}-\Psi^{2}\left(  W_{\varepsilon,\xi}+\phi_{2}\right)  \phi_{2}\right\vert
_{p^{\prime},\varepsilon}\\
= &  :I_{1}+I_{2}+I_{3}+I_{4}.
\end{align*}
By Remark \ref{remark:Weps} and Lemma \ref{lem:e7} with $s:=\frac{3}{2}$, for
some $\theta\in(0,1)$ we have that
\begin{align*}
I_{1}^{p^{\prime}} &  \leq\frac{C}{\varepsilon^{2}}\left(  \int_{M}\left\vert
\Psi^{\prime}\left(  W_{\varepsilon,\xi}+\theta\phi_{1}+(1-\theta)\phi
_{2}\right)  (\phi_{1}-\phi_{2})\right\vert ^{p}\right)  ^{\frac{p^{\prime}%
}{p}}\left(  \frac{1}{\varepsilon^{2}}\int_{M}\left\vert W_{\varepsilon,\xi
}\right\vert ^{\frac{p^{\prime}p}{p-p^{\prime}}}\right)  ^{\frac{p-p^{\prime}%
}{p}}\varepsilon^{\frac{2(p-p^{\prime})}{p}}\\
&  \leq C\frac{\varepsilon^{\frac{2(p-p^{\prime})}{p}}}{\varepsilon^{2}%
}\left(  \varepsilon^{\frac{4}{3}}+\Vert\phi_{1}\Vert_{g}+\Vert\phi_{2}%
\Vert_{g}\right)  ^{p^{\prime}}\Vert\phi_{1}-\phi_{2}\Vert_{g}^{p^{\prime}}\\
&  \leq Cl_{\varepsilon}\Vert\phi_{1}-\phi_{2}\Vert_{\varepsilon}^{p^{\prime}%
},
\end{align*}
for $\Vert\phi_{1}\Vert_{\varepsilon},\Vert\phi_{2}\Vert_{\varepsilon}\leq
R\varepsilon,$ with $l_{\varepsilon}:=\varepsilon^{\frac{p^{\prime}(p-2)}{p}%
}\rightarrow0$ as $\varepsilon\rightarrow0$. From the estimate of $I_{1}$,
recalling that $0\leq\Psi(u)\leq\frac{1}{q},$ we derive
\begin{align*}
I_{2}^{p^{\prime}} &  =\frac{1}{\varepsilon^{2}}\int_{M}\left\vert \Psi\left(
W_{\varepsilon,\xi}+\phi_{1}\right)  +\Psi\left(  W_{\varepsilon,\xi}+\phi
_{2}\right)  \right\vert ^{p^{\prime}}\left\vert \Psi\left(  W_{\varepsilon
,\xi}+\phi_{1}\right)  -\Psi\left(  W_{\varepsilon,\xi}+\phi_{2}\right)
\right\vert ^{p^{\prime}}\left\vert W_{\varepsilon,\xi}\right\vert
^{p^{\prime}}\\
&  \leq CI_{1}^{p^{\prime}}.
\end{align*}
On the other hand, choosing $\vartheta\in(1,2)$ in Lemma \ref{lem:e3} such
that $\vartheta p^{\prime}>2$ and applying Lemma \ref{lem:e7} with
$s:=\frac{3}{2}$, we obtain
\begin{align*}
I_{3}^{p^{\prime}} &  \leq\frac{1}{\varepsilon^{2}}\int_{M}\left\vert
\Psi^{\prime}\left(  W_{\varepsilon,\xi}+\theta\phi_{1}+(1-\theta)\phi
_{2}\right)  (\phi_{1}-\phi_{2})\right\vert ^{p^{\prime}}\left\vert \phi
_{1}\right\vert ^{p^{\prime}}\\
&  \qquad+\frac{1}{\varepsilon^{2}}\int_{M}\left\vert \Psi\left(
W_{\varepsilon,\xi}+\phi_{2}\right)  \right\vert ^{p^{\prime}}\left\vert
\phi_{1}-\phi_{2}\right\vert ^{p^{\prime}}\\
&  \leq C\frac{1}{\varepsilon^{2}}\left(  \int_{M}\left\vert \Psi^{\prime
}\left(  W_{\varepsilon,\xi}+\theta\phi_{1}+(1-\theta)\phi_{2}\right)
(\phi_{1}-\phi_{2})\right\vert ^{p}\right)  ^{\frac{p^{\prime}}{p}}\left(
\int_{M}\left\vert \phi_{1}\right\vert ^{\frac{p^{\prime}p}{p-p^{\prime}}%
}\right)  ^{\frac{p-p^{\prime}}{p}}\\
&  \qquad+C\frac{1}{\varepsilon^{2}}\left(  \int_{M}\left\vert \phi_{1}%
-\phi_{2}\right\vert ^{p}\right)  ^{\frac{p^{\prime}}{p}}\left(  \int
_{M}\left\vert \Psi\left(  W_{\varepsilon,\xi}+\phi_{2}\right)  \right\vert
^{\frac{p^{\prime}p}{p-p^{\prime}}}\right)  ^{\frac{p-p^{\prime}}{p}}\\
&  \leq C\frac{1}{\varepsilon^{2}}\left(  \varepsilon^{\frac{4}{3}}+\Vert
\phi_{1}\Vert_{g}+\Vert\phi_{2}\Vert_{g}\right)  ^{p^{\prime}}\Vert\phi
_{1}-\phi_{2}\Vert_{g}^{p^{\prime}}\Vert\phi_{1}\Vert_{g}^{p^{\prime}}\\
&  \qquad+C\frac{\varepsilon^{\vartheta p^{\prime}}}{\varepsilon^{2}}%
(1+\Vert\phi_{2}\Vert_{\varepsilon}^{2})\Vert\phi_{1}-\phi_{2}\Vert
_{g}^{p^{\prime}}\\
&  \leq C\left(  \frac{\varepsilon^{2p^{\prime}}}{\varepsilon^{2}}%
+\frac{\varepsilon^{\vartheta p^{\prime}}}{\varepsilon^{2}}\right)  \Vert
\phi_{1}-\phi_{2}\Vert_{\varepsilon}^{p^{\prime}}=l_{\varepsilon}\Vert\phi
_{1}-\phi_{2}\Vert_{\varepsilon}^{p^{\prime}},
\end{align*}
for $\Vert\phi_{1}\Vert_{\varepsilon},\Vert\phi_{2}\Vert_{\varepsilon}\leq
R\varepsilon,$ where $l_{\varepsilon}\rightarrow0$ as $\varepsilon
\rightarrow0$.

Finally, from the estimate of $I_{2}$ we derive $I_{4}^{p^{\prime}}\leq
CI_{3}^{p^{\prime}}$ Collecting the previous estimates we obtain \eqref{se12}.
\end{proof}

\begin{proof}
[\textbf{Proof of Proposition \ref{resto}}]From Proposition \ref{inv} we
deduce
\[
\left\Vert T_{\varepsilon,\xi}(\phi)\right\Vert _{\varepsilon}\leq C\left(
\left\Vert N_{\varepsilon,\xi}(\phi)\right\Vert _{\varepsilon}+\left\Vert
S_{\varepsilon,\xi}(\phi)\right\Vert _{\varepsilon}+\left\Vert R_{\varepsilon
,\xi}\right\Vert _{\varepsilon}\right)
\]
and
\[
\left\Vert T_{\varepsilon,\xi}(\phi_{1})-T_{\varepsilon,\xi}(\phi
_{2})\right\Vert _{\varepsilon}\leq C\left\Vert N_{\varepsilon,\xi}(\phi
_{1})-N_{\varepsilon,\xi}(\phi_{2})\right\Vert _{\varepsilon}+C\left\Vert
S_{\varepsilon,\xi}(\phi_{1})-S_{\varepsilon,\xi}(\phi_{2})\right\Vert
_{\varepsilon}.
\]
Lemmas \ref{re1}, \ref{ne1} and \ref{se1} imply that $T_{\varepsilon,\xi}$ is
a contraction in the ball centered at $0$ of radius $R\varepsilon$ in
$K_{\varepsilon,\xi}^{\perp},$ for a suitable constant $R.$ Hence,
$T_{\varepsilon,\xi}$ has a unique fixed point.

In order to prove that the map $\xi\mapsto\phi_{\varepsilon,\xi}$ is
$\mathcal{C}^{1}$ we apply the implicit function theorem to the $\mathcal{C}%
^{1}$-function $G:M\times H_{\varepsilon}\rightarrow H_{\varepsilon}$ defined
by
\begin{align*}
G(\xi,u) &  :=\Pi_{\varepsilon,\xi}^{\perp}\left\{  W_{\varepsilon,\xi}%
+\Pi_{\varepsilon,\xi}^{\perp}u-i_{\varepsilon}^{\ast}\left[  b(x)f\left(
W_{\varepsilon,\xi}+\Pi_{\varepsilon,\xi}^{\perp}u\right)  +\omega
^{2}b(x)g\left(  W_{\varepsilon,\xi}+\Pi_{\varepsilon,\xi}^{\perp}u\right)
\right]  \right\}  \\
&  \qquad+\Pi_{\varepsilon,\xi}u.
\end{align*}
Note that $G\left(  \xi,\phi_{\varepsilon,\xi}\right)  =0.$ Next we show that
the linearized operator $\frac{\partial G}{\partial u}\left(  \xi
,\phi_{\varepsilon,\xi}\right)  :H_{\varepsilon}\rightarrow H_{\varepsilon}$
defined by
\begin{align*}
&  \frac{\partial G}{\partial u}\left(  \xi,\phi_{\varepsilon,\xi}\right)
(u)\\
&  =\Pi_{\varepsilon,\xi}^{\perp}\left\{  \Pi_{\varepsilon,\xi}^{\perp
}(u)-i_{\varepsilon}^{\ast}\left[  b(x)f^{\prime}\left(  W_{\varepsilon,\xi
}+\phi_{\varepsilon,\xi}\right)  \Pi_{\varepsilon,\xi}^{\perp}(u)+\omega
^{2}b(x)g^{\prime}\left(  W_{\varepsilon,\xi}+\phi_{\varepsilon,\xi}\right)
\Pi_{\varepsilon,\xi}^{\perp}(u)\right]  \right\}  \\
&  \qquad+\Pi_{\varepsilon,\xi}(u)
\end{align*}
is invertible, provided $\varepsilon$ is small enough. For any $\phi$ with
$\left\Vert \phi\right\Vert _{\varepsilon}\leq C\varepsilon$ we have that
\begin{align*}
&  \left\Vert \frac{\partial G}{\partial u}\left(  \xi,\phi_{\varepsilon,\xi
}\right)  (u)\right\Vert _{\varepsilon}\geq C\left\Vert \Pi_{\varepsilon,\xi
}(u)\right\Vert _{\varepsilon}\\
&  \qquad+C\left\Vert \Pi_{\varepsilon,\xi}^{\perp}\left\{  \Pi_{\varepsilon
,\xi}^{\perp}(u)-i_{\varepsilon}^{\ast}\left[  f^{\prime}\left(
W_{\varepsilon,\xi}+\phi_{\varepsilon,\xi}\right)  \Pi_{\varepsilon,\xi
}^{\perp}(u)+\omega^{2}g^{\prime}\left(  W_{\varepsilon,\xi}+\phi
_{\varepsilon,\xi}\right)  \Pi_{\varepsilon,\xi}^{\perp}(u)\right]  \right\}
\right\Vert _{\varepsilon}\\
&  \geq C\left\Vert \Pi_{\varepsilon,\xi}(u)\right\Vert _{\varepsilon
}+C\left\Vert L_{\varepsilon,\xi}\left(  \Pi_{\varepsilon,\xi}^{\perp
}(u)\right)  \right\Vert _{\varepsilon}\\
&  \qquad-C\left\Vert \Pi_{\varepsilon,\xi}^{\perp}\left\{  i_{\varepsilon
}^{\ast}\left[  \left(  f^{\prime}\left(  W_{\varepsilon,\xi}+\phi
_{\varepsilon,\xi}\right)  -f^{\prime}\left(  W_{\varepsilon,\xi}\right)
\right)  \Pi_{\varepsilon,\xi}^{\perp}(u)\right]  \right\}  \right\Vert
_{\varepsilon}\\
&  \qquad-C\left\Vert \Pi_{\varepsilon,\xi}^{\perp}\left\{  i_{\varepsilon
}^{\ast}\left[  \omega^{2}g^{\prime}\left(  W_{\varepsilon,\xi}+\phi
_{\varepsilon,\xi}\right)  \Pi_{\varepsilon,\xi}^{\perp}(u)\right]  \right\}
\right\Vert _{\varepsilon}\\
&  \geq C\left\Vert \Pi_{\varepsilon,\xi}(u)\right\Vert _{\varepsilon
}+C\left\Vert \Pi_{\varepsilon,\xi}^{\perp}(u)\right\Vert _{\varepsilon
}-o(1)\left\Vert \Pi_{\varepsilon,\xi}^{\perp}(u)\right\Vert _{\varepsilon}\\
&  \geq C\left\Vert u\right\Vert _{\varepsilon}.
\end{align*}
Indeed, by (\ref{eq:stimaf}) we have
\begin{align*}
\left\Vert \Pi_{\varepsilon,\xi}^{\perp}\left\{  i_{\varepsilon}^{\ast}\left[
\left(  f^{\prime}\left(  W_{\varepsilon,\xi}+\phi_{\varepsilon,\xi}\right)
-f^{\prime}\left(  W_{\varepsilon,\xi}\right)  \right)  \Pi_{\varepsilon,\xi
}^{\perp}(u)\right]  \right\}  \right\Vert _{\varepsilon} &  \leq C\left(
\left\Vert \phi\right\Vert _{\varepsilon}^{p-2}+\left\Vert \phi\right\Vert
_{\varepsilon}\right)  \left\Vert \Pi_{\varepsilon,\xi}^{\perp}(u)\right\Vert
_{\varepsilon}\\
&  =o(1)\left\Vert \Pi_{\varepsilon,\xi}^{\perp}(u)\right\Vert _{\varepsilon}.
\end{align*}
Moreover,
\begin{align*}
&  \left\Vert \Pi_{\varepsilon,\xi}^{\perp}\left\{  i_{\varepsilon}^{\ast
}\left[  \omega^{2}g^{\prime}\left(  W_{\varepsilon,\xi}+\phi_{\varepsilon
,\xi}\right)  \Pi_{\varepsilon,\xi}^{\perp}(u)\right]  \right\}  \right\Vert
_{\varepsilon}\\
&  \leq C\left\vert \left(  W_{\varepsilon,\xi}+\phi_{\varepsilon,\xi}\right)
\left(  2q-2q^{2}\Psi\left(  W_{\varepsilon,\xi}+\phi_{\varepsilon,\xi
}\right)  \right)  \Psi^{\prime}\left(  W_{\varepsilon,\xi}+\phi
_{\varepsilon,\xi}\right)  \left[  \Pi_{\varepsilon,\xi}^{\perp}(u)\right]
\right\vert _{p^{\prime},\varepsilon}\\
&  \qquad+C\left\vert \left[  2q\Psi\left(  W_{\varepsilon,\xi}+\phi
_{\varepsilon,\xi}\right)  -q^{2}\Psi^{2}\left(  W_{\varepsilon,\xi}%
+\phi_{\varepsilon,\xi}\right)  \right]  \Pi_{\varepsilon,\xi}^{\perp
}(u)\right\vert _{p^{\prime},\varepsilon}\\
&  :=I_{1}+I_{2}.
\end{align*}
From Lemma \ref{lem:e7} we derive
\begin{align*}
I_{1} &  \leq\frac{C}{\varepsilon^{\frac{2}{p^{\prime}}}}\left\vert
W_{\varepsilon,\xi}+\phi_{\varepsilon,\xi}\right\vert _{g,2}\left\vert
\Psi^{\prime}\left(  W_{\varepsilon,\xi}+\phi_{\varepsilon,\xi}\right)
\Pi_{\varepsilon,\xi}^{\perp}(u)\right\vert _{g,\frac{4p^{\prime}}%
{2-p^{\prime}}}\left\vert 2q-2q^{2}\Psi\left(  W_{\varepsilon,\xi}%
+\phi_{\varepsilon,\xi}\right)  \right\vert _{g,\frac{4p^{\prime}}%
{2-p^{\prime}}}\\
&  \leq C\frac{1}{\varepsilon^{\frac{2}{p^{\prime}}}}\varepsilon
(\varepsilon^{\frac{4}{3}}+\varepsilon)\left\Vert \Pi_{\varepsilon,\xi}%
^{\perp}u\right\Vert _{g}\leq\varepsilon^{2-\frac{2}{p^{\prime}}}\left\Vert
\Pi_{\varepsilon,\xi}^{\perp}u\right\Vert _{g}=o(1)\left\Vert \Pi
_{\varepsilon,\xi}^{\perp}u\right\Vert _{g},
\end{align*}
and, since $0\leq\Psi(u)\leq1/q$, from Lemma \ref{lem:e3} with $\vartheta
p^{\prime}>2$ we get
\begin{align*}
I_{2} &  \leq\frac{C}{\varepsilon^{\frac{2}{p^{\prime}}}}\left\vert
\Pi_{\varepsilon,\xi}^{\perp}u\right\vert _{g,p}\left\vert \Psi\left(
W_{\varepsilon,\xi}+\phi_{\varepsilon,\xi}\right)  \right\vert _{g,\frac
{p^{\prime}p}{p-p^{\prime}}}\\
&  \leq C\frac{\varepsilon^{\vartheta}}{\varepsilon^{\frac{2}{p^{\prime}}}%
}\left(  1+\left\Vert \phi_{\varepsilon,\xi}\right\Vert _{\varepsilon}%
^{2}\right)  \left\Vert \Pi_{\varepsilon,\xi}^{\perp}u\right\Vert
_{g}=o(1)\left\Vert \Pi_{\varepsilon,\xi}^{\perp}u\right\Vert _{g}%
\end{align*}
This concludes the proof.
\end{proof}

\section{The reduced energy}

\label{sec:red-exp}This section is devoted to the proof of Proposition
\ref{lem:tool2}.

\begin{lemma}
\label{fine4}The following estimate
\begin{align}
\widetilde{I}_{\varepsilon}(\xi) &  =I_{\varepsilon}\left(  W_{\varepsilon
,\xi}+\phi_{\varepsilon,\xi}\right)  \label{fine41}\\
&  =I_{\varepsilon}\left(  W_{\varepsilon,\xi}\right)  +o(1)=J_{\varepsilon
}\left(  W_{\varepsilon,\xi}\right)  +\frac{\omega^{2}}{2}G_{\varepsilon
}\left(  W_{\varepsilon,\xi}\right)  +o\left(  1\right)  \nonumber
\end{align}
holds true $\mathcal{C}^{0}$-uniformly with respect to $\xi$ as $\varepsilon$
goes to zero. Moreover, setting $\xi(y):=\exp_{\xi}(y),$ $y\in B(0,r),$ we
have that
\begin{align*}
\left(  \frac{\partial}{\partial y_{h}}\widetilde{I}_{\varepsilon}%
(\xi(y))\right)  _{|_{y=0}} &  =\left(  \frac{\partial}{\partial y_{h}%
}I_{\varepsilon}\left(  W_{\varepsilon,\xi(y)}+\phi_{\varepsilon,\xi
(y)}\right)  \right)  _{|_{y=0}}\\
&  =\left(  \frac{\partial}{\partial y_{h}}I_{\varepsilon}\left(
W_{\varepsilon,\xi(y)}\right)  \right)  _{|_{y=0}}+o(1)\\
&  =\left(  \frac{\partial}{\partial y_{h}}J_{\varepsilon}\left(
W_{\varepsilon,\xi(y)}\right)  \right)  _{|_{y=0}}+\frac{\omega^{2}}{2}\left(
\frac{\partial}{\partial y_{h}}G_{\varepsilon}\left(  W_{\varepsilon,\xi
(y)}\right)  \right)  _{|_{y=0}}+o\left(  1\right)  ,
\end{align*}
$\mathcal{C}^{0}$-uniformly with respect to $\xi$ as $\varepsilon$ goes to zero.
\end{lemma}

\begin{proof}
In Lemma 5.1 of \cite{CGM} we have proved the following two estimates:
\[
J_{\varepsilon}\left(  W_{\varepsilon,\xi(y)}+\phi_{\varepsilon,\xi
(y)}\right)  -J_{\varepsilon}\left(  W_{\varepsilon,\xi(y)}\right)  =o(1),
\]%
\[
\left(  J_{\varepsilon}^{\prime}\left(  W_{\varepsilon,\xi(y)}+\phi
_{\varepsilon,\xi(y)}\right)  -J_{\varepsilon}^{\prime}\left(  W_{\varepsilon
,\xi(y)}\right)  \right)  \left[  \left(  \frac{\partial}{\partial y_{h}%
}W_{\varepsilon,\xi(y)}\right)  _{|_{y=0}}\right]  =o(1).
\]
To complete the proof we shall prove the the following three estimates:
\begin{equation}
G_{\varepsilon}\left(  W_{\varepsilon,\xi}+\phi_{\varepsilon,\xi}\right)
-G_{\varepsilon}\left(  W_{\varepsilon,\xi}\right)  =o(1), \label{new1}%
\end{equation}%
\begin{equation}
\left[  G_{\varepsilon}^{\prime}\left(  W_{\varepsilon,\xi_{0}}+\phi
_{\varepsilon,\xi_{0}}\right)  -G_{\varepsilon}^{\prime}\left(  W_{\varepsilon
,\xi_{0}}\right)  \right]  \left[  \left(  \frac{\partial}{\partial y_{h}%
}W_{\varepsilon,\xi(y)}\right)  _{|_{y=0}}\right]  =o(1), \label{new2}%
\end{equation}%
\begin{equation}
\left(  J_{\varepsilon}^{\prime}\left(  W_{\varepsilon,\xi(y)}+\phi
_{\varepsilon,\xi(y)}\right)  +\frac{\omega^{2}}{2}G_{\varepsilon}^{\prime
}\left(  W_{\varepsilon,\xi(y)}+\phi_{\varepsilon,\xi(y)}\right)  \right)
\left[  \frac{\partial}{\partial y_{h}}\phi_{\varepsilon,\xi(y)}\right]
=o(1). \label{new3}%
\end{equation}

We start with (\ref{new1}). For some $\theta\in\lbrack0,1]$ we have
\begin{align*}
&  G_{\varepsilon}\left(  W_{\varepsilon,\xi}+\phi_{\varepsilon,\xi}\right)
-G_{\varepsilon}\left(  W_{\varepsilon,\xi}\right) \\
&  =\frac{1}{\varepsilon^{2}}\int\limits_{M}b(x)\left[  \Psi\left(
W_{\varepsilon,\xi}+\phi_{\varepsilon,\xi}\right)  \left(  W_{\varepsilon,\xi
}+\phi_{\varepsilon,\xi}\right)  ^{2}-\Psi\left(  W_{\varepsilon,\xi}\right)
\left(  W_{\varepsilon,\xi}\right)  ^{2}\right] \\
&  =\frac{1}{\varepsilon^{2}}\int\limits_{M}b(x)\Psi^{\prime}\left(
W_{\varepsilon,\xi}+\theta\phi_{\varepsilon,\xi}\right)  [\phi_{\varepsilon
,\xi}]\left(  W_{\varepsilon,\xi}\right)  ^{2}\\
&  \qquad+\frac{1}{\varepsilon^{2}}\int\limits_{M}b(x)\Psi\left(
W_{\varepsilon,\xi}+\phi_{\varepsilon,\xi}\right)  \left[  2\phi
_{\varepsilon,\xi}W_{\varepsilon,\xi}+\phi_{\varepsilon,\xi}^{2}\right]
\end{align*}
Since $\left\Vert \phi_{\varepsilon,\xi}\right\Vert _{\varepsilon}\leq
C\varepsilon$, from Lemma \ref{lem:e7} we obtain (\ref{new1}).

Next, we prove (\ref{new2}). For some $\theta\in\lbrack0,1]$ we have
\begin{align*}
&  \left[  G_{\varepsilon}^{\prime}\left(  W_{\varepsilon,\xi_{0}}%
+\phi_{\varepsilon,\xi_{0}}\right)  -G_{\varepsilon}^{\prime}\left(
W_{\varepsilon,\xi_{0}}\right)  \right]  \left[  \left(  \frac{\partial
}{\partial y_{h}}W_{\varepsilon,\xi(y)}\right)  _{|_{y=0}}\right]  \\
&  \leq\frac{q}{2\varepsilon^{2}}\left\vert \int_{M}b(x)\left\{  \left[
2\Psi(W_{\varepsilon,\xi_{0}}+\phi_{\varepsilon,\xi_{0}})-\Psi(W_{\varepsilon
,\xi_{0}})\right]  -\left[  q\Psi^{2}(W_{\varepsilon,\xi_{0}}+\phi
_{\varepsilon,\xi_{0}})-q\Psi^{2}(W_{\varepsilon,\xi_{0}})\right]  \right\}
\right.  \\
&  \qquad\qquad\left.  \cdot W_{\varepsilon,\xi_{0}}\left(  \frac{\partial
}{\partial y_{h}}W_{\varepsilon,\xi(y)}\right)  _{|_{y=0}}\right\vert \\
&  \quad+\left\vert \frac{q}{2\varepsilon^{2}}\int_{M}2b(x)\left[
\Psi(W_{\varepsilon,\xi_{0}}+\phi_{\varepsilon,\xi_{0}})-q\Psi^{2}%
(W_{\varepsilon,\xi_{0}}+\phi_{\varepsilon,\xi_{0}})\right]  \phi
_{\varepsilon,\xi_{0}}\left(  \frac{\partial}{\partial y_{h}}W_{\varepsilon
,\xi(y)}\right)  _{|_{y=0}}\right\vert \\
&  \leq\left\vert \frac{q}{2\varepsilon^{2}}\int_{M}2b(x)\Psi^{\prime
}(W_{\varepsilon,\xi_{0}}+\theta\phi_{\varepsilon,\xi_{0}})(\phi
_{\varepsilon,\xi_{0}})W_{\varepsilon,\xi_{0}}\left(  \frac{\partial}{\partial
y_{h}}W_{\varepsilon,\xi(y)}\right)  _{|_{y=0}}\right\vert \\
&  \quad+\left\vert \frac{q}{\varepsilon^{2}}\int_{M}b(x)\Psi(W_{\varepsilon
,\xi_{0}}+\theta\phi_{\varepsilon,\xi_{0}})\Psi^{\prime}(W_{\varepsilon
,\xi_{0}}+\theta\phi_{\varepsilon,\xi_{0}})(\phi_{\varepsilon,\xi_{0}%
})W_{\varepsilon,\xi_{0}}\left(  \frac{\partial}{\partial y_{h}}%
W_{\varepsilon,\xi(y)}\right)  _{|_{y=0}}\right\vert \\
&  \quad+\left\vert \frac{q}{\varepsilon^{2}}\int_{M}b(x)\Psi(W_{\varepsilon
,\xi_{0}})\phi_{\varepsilon,\xi_{0}}\left(  \frac{\partial}{\partial y_{h}%
}W_{\varepsilon,\xi(y)}\right)  _{|_{y=0}}\right\vert \\
&  \quad+\left\vert \frac{q}{\varepsilon^{2}}\int_{M}b(x)\Psi^{\prime
}(W_{\varepsilon,\xi_{0}}+\theta\phi_{\varepsilon,\xi_{0}})(\phi
_{\varepsilon,\xi_{0}})\phi_{\varepsilon,\xi_{0}}\left(  \frac{\partial
}{\partial y_{h}}W_{\varepsilon,\xi(y)}\right)  _{|_{y=0}}\right\vert \\
&  \quad+\left\vert \frac{q}{2\varepsilon^{2}}\int_{M}b(x)\Psi^{2}%
(W_{\varepsilon,\xi_{0}}+\phi_{\varepsilon,\xi_{0}})(\phi_{\varepsilon,\xi
_{0}})\left(  \frac{\partial}{\partial y_{h}}W_{\varepsilon,\xi(y)}\right)
_{|_{y=0}}\right\vert \\
&  :=I_{1}+I_{2}+I_{3}+I_{4}+I_{5}%
\end{align*}
From Lemma \ref{lem:e7}, Remark \ref{remark:Weps} and equations
(\ref{eq:derWeps-1}), (\ref{eq:espE}), (\ref{eq:espg1}), (\ref{eq:espg2}),
recalling that $\left\Vert \phi_{\varepsilon,\xi(y)}\right\Vert _{\varepsilon
}\leq C\varepsilon,$ we get
\begin{align*}
I_{1} &  \leq C\frac{\varepsilon^{\frac{4}{3}}}{\varepsilon^{2}}\left(
\int_{M}\left[  \Psi^{\prime}(W_{\varepsilon,\xi_{0}}+\phi_{\varepsilon
,\xi_{0}})(\phi_{\varepsilon,\xi_{0}})\right]  ^{3}\right)  ^{\frac{1}{3}%
}\left(  \frac{1}{\varepsilon^{2}}\int_{M}W_{\varepsilon,\xi_{0}}^{3}\right)
^{\frac{1}{3}}\left(  \frac{1}{\varepsilon^{2}}\int_{M}\left[  \left(
\frac{\partial}{\partial y_{h}}W_{\varepsilon,\xi(y)}\right)  _{|_{y=0}%
}\right]  ^{3}\right)  ^{\frac{1}{3}}\\
&  \leq C\varepsilon^{\frac{4}{3}}\left(  \int_{\mathbb{R}^{2}}\left[
\sum_{k=1}^{2}\left\vert \frac{1}{\varepsilon}\frac{\partial U}{\partial
z_{k}}(z)\chi(\varepsilon z)+\left(  \chi(\varepsilon z)+\frac{\partial\chi
}{\partial z_{k}}(\varepsilon z)\right)  U(z)\right\vert \right]
^{3}dz\right)  ^{\frac{1}{3}}\\
&  \leq C\varepsilon^{\frac{4}{3}}\frac{1}{\varepsilon}=O(\varepsilon
^{\frac{1}{3}})
\end{align*}
In a similar way, using Lemma \ref{lem:e7} and embedding the first and the
second term in $L^{6}$ and the third one in $L^{3/2}$, we get
\[
I_{4}\leq C\frac{1}{\varepsilon^{2}}[\varepsilon^{4/3}\left\Vert
\phi_{\varepsilon,\xi}\right\Vert _{\varepsilon}+\left\Vert \phi
_{\varepsilon,\xi}\right\Vert _{\varepsilon}^{2}]\left\Vert \phi
_{\varepsilon,\xi}\right\Vert _{\varepsilon}\varepsilon^{\frac{4}{3}%
-1}=O(\varepsilon^{\frac{4}{3}}).
\]
For $I_{3}$ by Lemma \ref{lem:e3} we have
\begin{align*}
I_{3} &  \leq C\frac{\varepsilon^{\frac{4}{3}}}{\varepsilon^{2}}\left(
\int_{M}\left[  \Psi(W_{\varepsilon,\xi_{0}})\right]  ^{3}\right)  ^{\frac
{1}{3}}\left(  \frac{1}{\varepsilon^{2}}\int_{M}\phi_{\varepsilon,\xi_{0}}%
^{3}\right)  ^{\frac{1}{3}}\left(  \frac{1}{\varepsilon^{2}}\int_{M}\left[
\left(  \frac{\partial}{\partial y_{h}}W_{\varepsilon,\xi(y)}\right)
_{|_{y=0}}\right]  ^{3}\right)  ^{\frac{1}{3}}\\
&  \leq C\frac{\varepsilon^{\frac{4}{3}}}{\varepsilon^{2}}\Vert\Psi
(W_{\varepsilon,\xi_{0}})\Vert_{g}\Vert\phi_{\varepsilon,\xi_{0}}%
\Vert_{\varepsilon}\left(  \int_{\mathbb{R}^{2}}\left[  \sum_{k=1}%
^{2}\left\vert \frac{1}{\varepsilon}\frac{\partial U}{\partial z_{k}}%
(z)\chi(\varepsilon z)+\left(  \chi(\varepsilon z)+\frac{\partial\chi
}{\partial z_{k}}(\varepsilon z)\right)  U(z)\right\vert \right]
^{3}dz\right)  ^{\frac{1}{3}}\\
&  \leq C\frac{\varepsilon^{\frac{4}{3}}}{\varepsilon^{2}}\varepsilon
^{\frac{5}{3}}\varepsilon\frac{1}{\varepsilon}=O(\varepsilon)
\end{align*}
and, from the estimate for $I_{3}$, since $0<\Psi(W_{\varepsilon,\xi_{0}}%
+\phi_{\varepsilon,\xi_{0}})<1/q$, we obtain
\[
I_{5}\leq CI_{3}=O(\varepsilon).
\]

Finally, we prove (\ref{new3}). Following the proof of Lemma 5.1 in
\cite{CGM}, we need only to prove that
\[
\left\vert G_{\varepsilon}^{\prime}\left(  W_{\varepsilon,\xi(y)}%
+\phi_{\varepsilon,\xi(y)}\right)  [Z_{\varepsilon,\xi(y)}^{l}]\right\vert
=o(1),
\]
that is
\[
\left\vert \frac{1}{\varepsilon^{2}}\int_{M}\left[  \Psi(W_{\varepsilon
,\xi(y)}+\phi_{\varepsilon,\xi(y)})-q\Psi^{2}(W_{\varepsilon,\xi(y)}%
+\phi_{\varepsilon,\xi(y)})\right]  (W_{\varepsilon,\xi(y)}+\phi
_{\varepsilon,\xi(y)})Z_{\varepsilon,\xi(y)}^{l}\right\vert =o(1).
\]
We have
\begin{align*}
&  \left\vert \frac{1}{\varepsilon^{2}}\int_{M}\left[  \Psi(W_{\varepsilon
,\xi(y)}+\phi_{\varepsilon,\xi(y)})-q\Psi^{2}(W_{\varepsilon,\xi(y)}%
+\phi_{\varepsilon,\xi(y)})\right]  (W_{\varepsilon,\xi(y)}+\phi
_{\varepsilon,\xi(y)})Z_{\varepsilon,\xi(y)}^{l}\right\vert \\
&  \leq\frac{C}{\varepsilon^{2}}\int_{M}\left\vert \Psi(W_{\varepsilon,\xi
(y)}+\phi_{\varepsilon,\xi(y)})(W_{\varepsilon,\xi(y)}+\phi_{\varepsilon
,\xi(y)})Z_{\varepsilon,\xi(y)}^{l}\right\vert \\
&  \qquad+\frac{C}{\varepsilon^{2}}\int_{M}\left\vert \Psi^{2}(W_{\varepsilon
,\xi(y)}+\phi_{\varepsilon,\xi(y)})(W_{\varepsilon,\xi(y)}+\phi_{\varepsilon
,\xi(y)})Z_{\varepsilon,\xi(y)}^{l}\right\vert :=I_{1}+I_{2}.
\end{align*}
By Proposition \ref{prop:Zi}, we have that $\Vert Z_{\varepsilon,\xi(y)}%
^{l}\Vert_{\varepsilon}=O(1).$ So, by Lemma \ref{lem:e3} and Remark
\ref{remark:Weps}, we have
\begin{align*}
I_{1} &  \leq C\frac{\varepsilon^{\frac{4}{3}}}{\varepsilon^{2}}\left(
\int_{M}\left[  \Psi(W_{\varepsilon,\xi_{0}}+\phi_{\varepsilon,\xi_{0}%
})\right]  ^{3}\right)  ^{\frac{1}{3}}\left(  \frac{1}{\varepsilon^{2}}%
\int_{M}(W_{\varepsilon,\xi_{0}}+\phi_{\varepsilon,\xi_{0}})^{3}\right)
^{\frac{1}{3}}\left(  \frac{1}{\varepsilon^{2}}\int_{M}|Z_{\varepsilon,\xi
(y)}^{l}|^{3}\right)  ^{\frac{1}{3}}\\
&  \leq C\frac{\varepsilon^{\frac{4}{3}}}{\varepsilon^{2}}\Vert\Psi
(W_{\varepsilon,\xi_{0}}+\phi_{\varepsilon,\xi_{0}})\Vert_{g}\left(  \Vert
W_{\varepsilon,\xi_{0}}\Vert_{3,\varepsilon}+\Vert\phi_{\varepsilon,\xi_{0}%
}\Vert_{\varepsilon}\right)  \Vert Z_{\varepsilon,\xi(y)}^{l}\Vert
_{\varepsilon}=O(\varepsilon).
\end{align*}
Again, as $0<\Psi(W_{\varepsilon,\xi_{0}}+\phi_{\varepsilon,\xi_{0}})<1/q$, we
obtain
\[
I_{2}\leq CI_{1}=O(\varepsilon).
\]
This concludes the proof.
\end{proof}

\begin{lemma}
\label{lem:expJeps}The expansion
\[
I_{\varepsilon}(W_{\varepsilon,\xi})=\left(  \frac{1}{2}-\frac{1}{p}\right)
\frac{c(\xi)^{\frac{n}{2}}a(\xi)^{\frac{p}{p-2}-\frac{n}{2}}}{b(\xi)^{\frac
{2}{p-2}}}\int_{\mathbb{R}^{n}}U^{p}dz+o(1)
\]
holds true $\mathcal{C}^{1}$-uniformly with respect to $\xi\in M$.
\end{lemma}

\begin{proof}
In Lemma 5.2 of \cite{CGM} we proved that
\[
J_{\varepsilon}(W_{\varepsilon,\xi})=\left(  \frac{1}{2}-\frac{1}{p}\right)
\frac{c(\xi)^{\frac{n}{2}}a(\xi)^{\frac{p}{p-2}-\frac{n}{2}}}{b(\xi)^{\frac
{2}{p-2}}}\int_{\mathbb{R}^{n}}U^{p}dz+O(\varepsilon).
\]
Hence, it suffices to show now that $\left\vert G_{\varepsilon}(W_{\varepsilon
,\xi})\right\vert =o(1)$, $\mathcal{C}^{1}$-uniformly with respect to $\xi\in
M$.

Regarding the $\mathcal{C}^{0}$-convergence, by Remark \ref{remark:Weps} and
Lemma \ref{lem:e3}, we have that
\begin{align*}
\left\vert G_{\varepsilon}(W_{\varepsilon,\xi})\right\vert  &  \leq\frac
{C}{\varepsilon^{2}}\int_{M}\Psi(W_{\varepsilon,\xi})W_{\varepsilon,\xi}%
^{2}d\mu_{g}\\
&  \leq C\frac{\varepsilon}{\varepsilon^{2}}\left(  \int_{M}\Psi
(W_{\varepsilon,\xi})^{2}\right)  ^{\frac{1}{2}}\left(  \frac{1}%
{\varepsilon^{2}}\int_{M}W_{\varepsilon,\xi}^{4}\right)  ^{\frac{1}{2}}\\
&  \leq C\frac{1}{\varepsilon}\Vert\Psi(W_{\varepsilon,\xi})\Vert_{g}\leq
\frac{\varepsilon^{\frac{5}{3}}}{\varepsilon}=O(\varepsilon^{\frac{2}{3}}).
\end{align*}
Regarding the $\mathcal{C}^{1}$-convergence observe that
\begin{align*}
\left\vert \left.  \frac{\partial}{\partial y_{h}}G_{\varepsilon
}(W_{\varepsilon,\xi})\right\vert _{y=0}\right\vert  &  \leq\left\vert
\frac{C}{\varepsilon^{2}}\left.  \frac{\partial}{\partial y_{h}}\int_{M}%
\Psi(W_{\varepsilon,\xi(y)})W_{\varepsilon,\xi(y)}^{2}\right\vert _{y=0}%
d\mu_{g}\right\vert \\
&  \leq\left\vert \left.  \frac{C}{\varepsilon^{2}}\int_{M}\Psi(W_{\varepsilon
,\xi(y)})2W_{\varepsilon,\xi(y)}\left(  \frac{\partial}{\partial y_{h}%
}W_{\varepsilon,\xi(y)}\right)  \right\vert _{y=0}d\mu_{g}\right\vert \\
&  \qquad+\left\vert \frac{C}{\varepsilon^{2}}\int_{M}W_{\varepsilon,\xi
(y)}^{2}\Psi^{\prime}(W_{\varepsilon,\xi(y)})\left[  \left.  \frac{\partial
}{\partial y_{h}}W_{\varepsilon,\xi(y)}\right\vert _{y=0}\right]  d\mu
_{g}\right\vert \\
&  :=I_{1}+I_{2}.
\end{align*}
Now, from Remark \ref{remark:Weps}, Lemma \ref{lem:e3}, and the estimates
(\ref{eq:derWeps-1}) and (\ref{eq:espE}), we derive
\begin{align*}
I_{1} &  \leq C\frac{\varepsilon^{\frac{8}{5}}}{\varepsilon^{2}}\left(
\int_{M}\Psi(W_{\varepsilon,\xi(y)})^{5}\right)  ^{\frac{1}{5}}\left(
\frac{1}{\varepsilon^{2}}\int_{M}W_{\varepsilon,\xi(y)}^{\frac{5}{2}}\right)
^{\frac{2}{5}}\left(  \frac{1}{\varepsilon^{2}}\int_{M}\left(  \left.  \left(
\frac{\partial}{\partial y_{h}}W_{\varepsilon,\xi(y)}\right)  \right\vert
_{y=0}\right)  ^{\frac{5}{2}}\right)  ^{\frac{2}{5}}\\
&  \leq C\frac{\varepsilon^{\frac{8}{5}}}{\varepsilon^{2}}\varepsilon
^{\frac{8}{5}}\frac{1}{\varepsilon}=o(1).
\end{align*}
On the other hand, from Remark \ref{remark:Weps}, the proof of Lemma
\ref{lem:e7}, and the estimates (\ref{eq:derWeps-1}) and (\ref{eq:espE}), for
some $t\in(1,3/2)$ we obtain%
\begin{align*}
I_{2} &  \leq C\frac{\varepsilon^{\frac{2}{t}}}{\varepsilon^{2}}\left(
\frac{1}{\varepsilon^{2}}\int_{M}W_{\varepsilon,\xi(h)}^{2t}\right)
^{\frac{1}{t}}\left(  \int_{M}\left(  \Psi^{\prime}(W_{\varepsilon,\xi
(y)})\left[  \left.  \frac{\partial}{\partial y_{h}}W_{\varepsilon,\xi
(h)}\right\vert _{y=0}\right]  \right)  ^{t^{\prime}}\right)  ^{\frac
{1}{t^{\prime}}}\\
&  \leq C\frac{\varepsilon^{\frac{2}{t}}}{\varepsilon^{2}}\left\Vert
\Psi^{\prime}(W_{\varepsilon,\xi(y)})\left[  \left.  \frac{\partial}{\partial
y_{h}}W_{\varepsilon,\xi(h)}\right\vert _{y=0}\right]  \right\Vert _{g}\\
&  \leq C\frac{\varepsilon^{\frac{2}{t}}}{\varepsilon^{2}}\varepsilon
^{\frac{4}{3}}\left\vert \left.  \frac{\partial}{\partial y_{h}}%
W_{\varepsilon,\xi(h)}\right\vert _{y=0}\right\vert _{g,6}\\
&  \leq C\frac{\varepsilon^{\frac{2}{t}}}{\varepsilon^{2}}\varepsilon
^{\frac{4}{3}}\varepsilon^{\frac{1}{3}}\left(  \frac{1}{\varepsilon^{2}}%
\int_{M}\left(  \left.  \frac{\partial}{\partial y_{h}}W_{\varepsilon,\xi
(h)}\right\vert _{y=0}\right)  ^{6}\right)  ^{\frac{1}{6}}\\
&  \leq C\frac{\varepsilon^{\frac{2}{t}}}{\varepsilon^{2}}\varepsilon
^{\frac{4}{3}}\varepsilon^{\frac{1}{3}}\frac{1}{\varepsilon}=C\varepsilon
^{\frac{2}{t}-\frac{4}{3}}=o(1)\text{.}%
\end{align*}
This concludes the proof.
\end{proof}

\section{Some estimates involving $\Psi$}

\label{sec:tech}

We start by pointing out the following facts.

\begin{remark}
\label{rem:17}There exists a constant $C>0$ such that, for every $\varphi\in
H_{g}^{1}(M)$ and every $0<\varepsilon<1,$ we have
\begin{align*}
C\Vert\varphi\Vert_{g}^{2} &  =C\int_{M}\left(  |\nabla_{g}\varphi
|^{2}+\varphi^{2}\right)  d\mu_{g}\\
&  \leq\int_{M}\left(  c(x)|\nabla_{g}\varphi|^{2}+\frac{d(x)}{\varepsilon
^{2}}\varphi^{2}\right)  d\mu_{g}=\Vert\varphi\Vert_{\varepsilon}^{2}.
\end{align*}

\end{remark}

\begin{remark}
\label{remark:Weps}The following estimates
\[
\lim_{\varepsilon\rightarrow0}\frac{1}{\varepsilon^{2}}\left\vert
W_{\varepsilon,\xi}\right\vert _{g,p}^{p}\leq C|U|_{p}^{p},\qquad p\geq2,
\]%
\[
\lim_{\varepsilon\rightarrow0}\left\vert \nabla_{g}W_{\varepsilon,\xi
}\right\vert _{g,2}^{2}\leq C|\nabla U|_{2}^{2}%
\]
hold true uniformly with respect to $\xi\in M$.
\end{remark}

Abusing notation we write%
\[
\Vert u\Vert_{g}^{2}=\int_{M}\left(  c(x)|\nabla_{g}\varphi|^{2}%
+b(x)u^{2}\right)  d\mu_{g}.
\]
This norm is equivalent to the standard norm (\ref{norms}) of $H_{g}^{1}(M).$
From equations (\ref{eq:e1}), (\ref{psipos}) and (\ref{eq:e2}) we obtain%
\begin{align}
\Vert\Psi(u)\Vert_{g}^{2} &  =\int_{M}b(x)qu^{2}\Psi(u)d\mu_{g}-\int
_{M}b(x)q^{2}u^{2}\left(  \Psi(u)\right)  ^{2}d\mu_{g}\label{psi1}\\
&  \leq C\int_{M}u^{2}\Psi(u)d\mu_{g},\nonumber
\end{align}%
\begin{align}
\Vert\Psi^{\prime}(u)\left[  h\right]  \Vert_{g}^{2} &  =\int_{M}%
2b(x)qu(1-q\Psi(u))h\Psi^{\prime}(u)\left[  h\right]  d\mu_{g}\label{psi2}\\
&  \qquad-\int_{M}b(x)q^{2}u^{2}\left(  \Psi^{\prime}(u)\left[  h\right]
\right)  ^{2}d\mu_{g}\nonumber\\
&  \leq C\int_{M}\left\vert u\right\vert \left\vert h\right\vert \left\vert
\Psi^{\prime}(u)\left[  h\right]  \right\vert d\mu_{g},\nonumber
\end{align}
for all $u,h\in H_{g}^{1}(M).$

\begin{lemma}
\label{lem:e3}Given $\vartheta\in(1,2)$ there is a constant $C>0$ such that
the inequality
\[
\Vert\Psi(W_{\varepsilon,\xi}+\varphi)\Vert_{g}\leq C(\varepsilon^{\vartheta
}+\Vert\varphi\Vert_{g}^{2})
\]
holds true for every $\varphi\in H_{g}^{1}(M)$, $\xi\in M$ and small enough
$\varepsilon>0.$
\end{lemma}

\begin{proof}
Let $t\in(2,\infty)$ be such that $\frac{2}{t^{\prime}}=\vartheta$ where
$t^{\prime}$ is the exponent conjugate to $t.$ From inequality (\ref{psi1}) we
obtain
\begin{align*}
\Vert\Psi(W_{\varepsilon,\xi}+\varphi)\Vert_{g}^{2} &  \leq C\left(  \int
_{M}\left[  \Psi\left(  W_{\varepsilon,\xi}+\varphi\right)  \right]  ^{t}%
d\mu_{g}\right)  ^{1/t}\left(  \int_{M}(W_{\varepsilon,\xi}+\varphi
)^{2t^{\prime}}\right)  ^{1/t^{\prime}}\\
&  \leq C\Vert\Psi(W_{\varepsilon,\xi}+\varphi)\Vert_{g}\left\vert
W_{\varepsilon,\xi}+\varphi\right\vert _{g,2t^{\prime}}^{2}.
\end{align*}
Thus, by Remark \ref{remark:Weps},
\begin{align*}
\Vert\Psi(W_{\varepsilon,\xi}+\varphi)\Vert_{g} &  \leq C\left(
\varepsilon^{2/t^{\prime}}\left(  \frac{1}{\varepsilon^{2}}\int_{M}%
W_{\varepsilon,\xi}^{2t^{\prime}}\right)  ^{1/t^{\prime}}+\left(  \int
_{M}\varphi^{2t^{\prime}}\right)  ^{1/t^{\prime}}\right)  \\
&  \leq C(\varepsilon^{\vartheta}+\Vert\varphi\Vert_{g}^{2}),
\end{align*}
as claimed.
\end{proof}

\begin{lemma}
\label{lem:e7}Given $s\in(1,2)$ there is a constant $C>0$ such that the
inequality
\[
\Vert\Psi^{\prime}(W_{\varepsilon,\xi}+k)[h]\Vert_{g}\leq C\Vert h\Vert
_{g}\left(  \varepsilon^{\frac{2}{s}}+\Vert k\Vert_{g}\right)
\]
holds true for every $k,h\in H_{g}^{1}(M)$, $\xi\in M$ and small enough
$\varepsilon>0.$
\end{lemma}

\begin{proof}
From inequality (\ref{psi2}) we obtain,
\begin{align*}
\Vert\Psi^{\prime}(W_{\varepsilon,\xi}+k)[h]\Vert_{g}^{2} &  \leq C\int
_{M}\left\vert W_{\varepsilon,\xi}+k\right\vert \left\vert h\right\vert
\left\vert \Psi^{\prime}(W_{\varepsilon,\xi}+k)\left[  h\right]  \right\vert
d\mu_{g}\\
&  \leq C\left(  \int_{M}\left\vert W_{\varepsilon,\xi}\right\vert \left\vert
h\right\vert \left\vert \Psi^{\prime}(W_{\varepsilon,\xi}+k)\left[  h\right]
\right\vert d\mu_{g}+\int_{M}\left\vert k\right\vert \left\vert h\right\vert
\left\vert \Psi^{\prime}(W_{\varepsilon,\xi}+k)\left[  h\right]  \right\vert
d\mu_{g}\right)  \\
&  =:I_{1}+I_{2}.
\end{align*}
Set $t:=2s^{\prime}\in(4,\infty),$ where $s^{\prime}$ is the conjugate
exponent to $s.$ Using Remark \ref{remark:Weps} we conclude that%
\begin{align*}
I_{1} &  \leq C\left\vert \Psi^{\prime}(W_{\varepsilon,\xi}+k)[h]\right\vert
_{g,t}|h|_{g,t}\left\vert W_{\varepsilon,\xi}\right\vert _{g,s}\\
&  =C\Vert\Psi^{\prime}(W_{\varepsilon,\xi}+k)[h]\Vert_{g}\Vert h\Vert
_{g}\varepsilon^{\frac{2}{s}}\left(  \frac{1}{\varepsilon^{2}}\int
_{M}W_{\varepsilon,\xi}^{s}\right)  ^{1/s}\\
&  =C\Vert\Psi^{\prime}(W_{\varepsilon,\xi}+k)[h]\Vert_{g}\Vert h\Vert
_{g}\varepsilon^{\frac{2}{s}}.
\end{align*}
Since%
\[
I_{2}\leq C\left\vert \Psi^{\prime}(W_{\varepsilon,\xi}+k)[h]\right\vert
_{g,3}|h|_{g,3}\left\vert k\right\vert _{g,3}\leq C\Vert\Psi^{\prime
}(W_{\varepsilon,\xi}+k)[h]\Vert_{g}\Vert h\Vert_{g}\Vert k\Vert_{g},
\]
the claim follows.
\end{proof}

\begin{lemma}
\label{lem:e5}Consider the functions
\[
\tilde{v}_{\varepsilon,\xi}(z):=\left\{
\begin{array}
[c]{ll}%
\Psi(W_{\varepsilon,\xi})\left(  \exp_{\xi}(\varepsilon z)\right)  & \text{
for }z\in B(0,r/\varepsilon),\\
0 & \text{ for }z\in\mathbb{R}^{2}\smallsetminus B(0,r/\varepsilon).
\end{array}
\right.
\]
Then, for any $\vartheta\in(1,2)$, there exists a constant $C>0$, independent
of $\varepsilon,\xi$, such that
\begin{align*}
\left\vert \tilde{v}_{\varepsilon,\xi}(z)\right\vert _{L^{2}(\mathbb{R}^{3})}
&  \leq C\varepsilon^{\vartheta-1},\\
\left\vert \nabla\tilde{v}_{\varepsilon,\xi}(z)\right\vert _{L^{2}%
(\mathbb{R}^{3})}  &  \leq C\varepsilon^{\vartheta}.
\end{align*}

\end{lemma}

\begin{proof}
After a change of variables we have that
\begin{multline*}
\int_{B_{g}(\xi,r)}|\nabla\Psi(W_{\varepsilon,\xi})|^{2}+|\Psi(W_{\varepsilon
,\xi})|^{2}d\mu_{g}\\
=\varepsilon^{2}\int_{B(0,r/\varepsilon)}|g_{\xi}(\varepsilon z)|^{1/2}\left(
\sum_{ij}g_{\xi}^{ij}(\varepsilon z)\frac{1}{\varepsilon^{2}}\frac
{\partial\tilde{v}_{\varepsilon,\xi}(z)}{\partial z_{i}}\frac{\partial
\tilde{v}_{\varepsilon,\xi}(z)}{\partial z_{i}}+\tilde{v}_{\varepsilon,\xi
}^{2}(z)\right)  dz.
\end{multline*}
Thus
\[
\Vert\Psi(W_{\varepsilon,\xi})\Vert_{g}^{2}\geq C(\left\vert \nabla\tilde
{v}_{\varepsilon,\xi}\right\vert _{L^{2}(\mathbb{R}^{3})}^{2}+\varepsilon
^{2}\left\vert \tilde{v}_{\varepsilon,\xi}\right\vert _{L^{2}(\mathbb{R}^{3}%
)}^{2}).
\]
This, combined with Lemma \ref{lem:e3}, gives
\[
\left\vert \nabla\tilde{v}_{\varepsilon,\xi}\right\vert _{L^{2}(\mathbb{R}%
^{3})}+\varepsilon\left\vert \tilde{v}_{\varepsilon,\xi}\right\vert
_{L^{2}(\mathbb{R}^{3})}\leq C\varepsilon^{\vartheta},
\]
as claimed.
\end{proof}

\end{document}